\documentclass[12pt,leqno]{article}  
\usepackage{amsmath,amssymb,bbm}
\usepackage{amsthm,array,mathdots}
\usepackage{titlesec} 
\usepackage[all,cmtip]{xy}
\usepackage{blkarray,arydshln} 

\titlespacing{\section}{0cm}{3.5pc}{1.5pc}

\makeatletter
\def\@citex[#1]#2{\if@filesw\immediate\write\@auxout{\string\citation{#2}}\fi
  \def\@citea{}\@cite{\@for\@citeb:=#2\do
    {\@citea\def\@citea{\@citesep}\@ifundefined
       {b@\@citeb}{{\bf ?}\@warning
       {Citation `\@citeb' on page \thepage \space undefined}}%
{\csname b@\@citeb\endcsname}}}{#1}}
\def\@citesep{; }
\makeatother

\newtheoremstyle{Kang}{}{}{\itshape}{}{\bf}{}{.5em}{}
\theoremstyle{Kang}
\newtheorem{theorem}{Theorem}[section]
\renewcommand{\thetheorem}{\arabic{section}.\arabic{theorem}}
\newtheorem{lemma}[theorem]{Lemma}

\newtheorem{prop}[theorem]{Proposition}

\newtheoremstyle{Kremark}{}{}{}{}{\bf}{}{.5em}{}
\theoremstyle{Kremark}
\newtheorem*{remark}{Remark.}
\newtheorem{defn}[theorem]{Definition}

\newtheorem{other}{}

\newenvironment{Case}[1]{\medskip {\it Case #1.}}{}

\allowdisplaybreaks[1]  

\def\fn#1{\operatorname{#1}} 
\def\bm#1{\mathbbm{#1}}
\def\c#1{\mathcal{#1}}

\title{Class Numbers and Algebraic Tori}
\author{Akinari Hoshi$^{(1)}$, Ming-chang Kang$^{(2)}$ and Aiichi Yamasaki$^{(3)}$ \\[3mm]
\begin{minipage}{16cm} \begin{description} \itemsep=-1pt
\item[] $^{(1)}$Department of Mathematics, Niigata University,
Niigata, Japan,\\ E-mail: hoshi@math.sc.niigata-u.ac.jp \item[]
$^{(2)}$Department of Mathematics, National Taiwan University,
Taipei,\\ Taiwan, E-mail: kang@math.ntu.edu.tw \item[]
$^{(3)}$Department of Mathematics, Kyoto University, Kyoto,
Japan,\\ E-mail: aiichi.yamasaki@gmail.com
\end{description} \end{minipage}}
\date{}

\begin{document}

\maketitle

\footnote{\textit{\!\!\! $2010$ Mathematics Subject
Classification}. 14E08, 11R33, 20C10, 11R29.}
\footnote{\textit{\!\!\! Keywords and phrases}. Class number,
cyclotomic fields, algebraic torus, rationality problem, integral
representation.} \footnote{\!\!\! This work was partially
supported by JSPS KAKENHI Grant Numbers 24540019, 25400027.}

\footnote{\!\!\! Parts of the work were finished when the
first-named author and the third-named author were visiting the
National Center for Theoretic Sciences (Taipei Office), whose
support is gratefully acknowledged.}

\begin{abstract}
{\noindent\bf Abstract.} Let $p$ be an odd prime number, $D_p$ be
the dihedral group of order $2p$, $h_p$ and $h^+_p$ be the class
numbers of $\bm{Q}(\zeta_p)$ and $\bm{Q}(\zeta_p+ \zeta_p^{-1})$
respectively. Theorem. $h_p^+=1$ if and only if, for any field $k$
admitting a $D_p$-extension, all the algebraic $D_p$-tori over $k$
are stably rational. A similar result for $h_p=1$ and $C_p$-tori
is valid also.
\end{abstract}

\newpage
\section{Introduction}
The initial goal of our investigation was to understand some
rationality problem of algebraic tori defined over a non-closed
field. Unexpectedly, we arrived at a result which related the
rationality problem to a  criterion of $h_p^{+}=1$ where $h_p^{+}$
is the second factor of the class number $h_p$. According to
\cite[page 420]{Wa}, it is notoriously difficult to determine
$h_p^{+}$. Before stating the main results, let's recall some
terminology first.

Let $k\subset L$ be a field extension. The field $L$ is rational
over $k$ (in short, $k$-rational) if, for some $n$, $L\simeq
k(X_1,\ldots,X_n)$, the rational function field of $n$ variables
over $k$. $L$ is called stably rational over $k$ (or, stably
$k$-rational) if the field $L(Y_1,\ldots,Y_m)$ is $k$-rational
where $Y_1,\ldots,Y_m$ are some elements algebraically independent
over $L$. When $k$ is an infinite field, $L$ is called retract
$k$-rational, if there exist an affine domain $A$ whose quotient
field is $L$ and $k$-algebra morphisms $\varphi: A\to
k[X_1,\ldots,X_n][1/f]$, $\psi: k[X_1,\ldots,X_n][1/f]\to A$
satisfying $\psi\circ\varphi=1_A$ the identity map on $A$ where
$k[X_1,\ldots,X_n]$ is a polynomial ring over $k$, $f\in
k[X_1,\ldots,X_n]\backslash \{0\}$ \cite[Definition 3.1; Ka2,
Definition 1.1]{Sa}.

It is known that ``$k$-rational" $\Rightarrow$ ``stably
$k$-rational" $\Rightarrow$ ``retract $k$-rational". Moreover, if
$k$ is an algebraic number field, retract $k$-rationality  of
$k(G)$ implies the inverse Galois problem for the field $k$ and
the group $G$ \cite [Ka2]{Sa} (see Definition \ref{d6.8} for the
definition of $k(G)$).

\medskip
Let $k$ be a field, $G$ be a finite group. We will say that the
field $k$ admits a $G$-extension if there is a Galois field
extension $K/k$ such that $G \simeq Gal(K/k)$.

Let $k$ be a non-closed field. An algebraic torus $T$ defined over
$k$ is an algebraic group such that $T\times_{\fn{Spec}(k)}
\fn{Spec}(K)\simeq \bm{G}_{m,K}^d$ for some finite separable
extension $K/k$, for some positive integer $d$ where
$\bm{G}_{m,K}$ is the 1-dimensional multiplicative group defined
over $K$; the field $K$ is called a splitting field of $T$.

\begin{defn} \label{d1.9}
Let $G$ be a finite group, $k$ be a field admitting a
$G$-extension. An algebraic torus $T$ over $k$ is called a
$G$-torus if it has a splitting field $K$ which is Galois over $k$
with $Gal(K/k) \simeq G$.
\end{defn}

An algebraic torus $T$ over $k$ is $k$-rational (resp. stably
$k$-rational, retract $k$-rational) if so is its function field
$k(T)$ over $k$.

The birational classification of algebraic tori was studied by
Voskresenskii, Endo and Miyata, Chistov, Kunyavskii,
Colliot-Th\'el\`ene and Sansuc, Klyachko, etc. For a survey, see
\cite{Vo,Ku4}.

In particular, Voskresenskii proves that all the 2-dimensional
algebraic tori are rational \cite[page 57]{Vo}, and the birational
classification of 3-dimensional algebraic tori was solved by
Kunyavskii \cite{Ku3}.

\begin{defn} \label{d1.3}
Let $n$ be a positive integer, $\zeta_n$ be a primitive $n$-th
root of unity. Let $h_n$ and $h_n^+$ be the class numbers of
$\bm{Q}(\zeta_n)$ and $\bm{Q}(\zeta_n+\zeta_n^{-1})$ respectively.
It is known that $h_n^+$ divides $h_n$ (see \cite[p.40, Theorem
4.14]{Wa}). Thus we may write $h_n=h_n^+\cdot h_n^-$ where $h_n^-$
is a positive integer. The integers $h_n^-$ and $h_n^+$ are called
the first factor and the second factor of $h_n$ respectively.
\end{defn}

\bigskip
In the sequel we denote by $C_n$ and $D_n$ the cyclic group of
order $n$ and the dihedral group of order $2n$ respectively.

\begin{theorem}[{\cite[page 64]{Vo}}] \label{t1.1}

\begin{enumerate}
\item[{\rm (1)}] Let $k$ be a field admitting a $C_p$-extension
where $p$ is a prime number with $h_p =1$. Then all the $C_p$-tori
defined over $k$ are $k$-rational. \item[{\rm (2)}] Let $k$ be a
field admitting a $C_4$-extension. Then all the $C_4$-tori over
$k$ are $k$-rational.
\end{enumerate}
\end{theorem}

\begin{theorem} \label{t1.2}
\begin{enumerate}
\item[{\rm (1)}] {\rm (Kunyavskii \cite{Ku1})} Let $G=C_2\times
C_2$, the Klein four group, and $k$ be a field admitting a
$G$-extension. Then a $G$-tori over $k$ is stably $k$-rational if
and only if it is $k$-rational. \item[{\rm (2)}] {\rm (Kunyavskii
\cite{Ku2})} Let $S_3$ be the symmetric group of degree 3, and $k$
be a field admitting an $S_3$-extension. Then all the $S_3$-tori
over $k$ are $k$-rational.
\end{enumerate}
\end{theorem}

The main results of this paper are about the stable rationality of
$D_p$-tori.

\begin{theorem} \label{t1.4}
Let $p$ be an odd prime number, and $k$ be a field admitting a
$D_p$-extension. Then $h_p^+=1$ if and only if all the $D_p$-tori
over $k$ are stably $k$-rational.
\end{theorem}

According to Washington \cite[p.420]{Wa}, the calculation of
$h_p^+$ is rather sophiscated. It is known that $h_p^+=1$ if $p\le
67$; if the generalized Riemann hypothesis is assumed, then
$h_p^+=1$ if $p\le 157$, and $h_{163}^+=4$ \cite[page 421]{Wa}.

The proof of Theorem \ref{t1.4} will be given in Theorem
\ref{t4.9} and Theorem \ref{t5.4}. Using the same idea of the
proof of Theorem \ref{t5.4}, we may provide a supplement to Part
(1) of Theorem \ref{t1.1} as follows (see Theorem \ref{t6.6}).

\begin{theorem} \label{t1.5}
Let $p$ be a prime number, and $k$ be a field admitting a
$C_p$-extension. Then $h_p=1$ if and only if all the $C_p$-tori
over $k$ are $k$-rational (resp. stably $k$-rational).
\end{theorem}

The main idea of our proof is to reduce the rationality problem of
a $G$-torus over a field $k$ to that of its function field. If $M$
is the character module of $T$, i.e.\ $M=\fn{Hom}(T\otimes_k
K,\bm{G}_{m,K})$ , then the function field of $T$ is $K(M)^G$
where $K/k$ is a Galois extension with $Gal(K/k) \simeq G$ and $M$
is a $G$-lattice from its definition (see \cite[page 36]{Sw2}). On
the other hand, by Theorem \ref{t2.6}, the fixed field $K(M)^G$ is
stably rational over $k$ if and only if the flabby class of $M$,
$[M]^{fl}$, is a permutation lattice. Thus it suffices to study
which flabby lattices are stably permutation. The proof consists
of two ingredients: the classification of integral representations
of $D_p$ and the Steinitz class of an integral representation of
$C_p$.

Theorem \ref{t3.4} and Theorem \ref{t3.5} show that the
Krull-Schmidt-Azumaya Theorem fails in the case of integral
representations \cite [page 128]{CR}. The failure is not so
desperate at first sight; on the contrary, it becomes a crucial
step in proving some $D_p$-lattices are stably permutation. Such a
phenomenon was observed for the case $p=3$ and $p=5$ when we
studied lower-rank lattices \cite{HY}. By painstaking computer
experiments, we are led to the most general case, which is
recorded in Theorem \ref{t3.4}, Theorem \ref{t3.5} and Theorem
\ref{t3.6}. The final result is summarized in Theorem \ref{t4.11}.

On the other hand, we use the Steinitz class to detect whether a
flabby lattice is stably permutation or not, when we regard a
$D_p$-lattice as a $C_p$-lattice by restriction (for the Steinitz
class, see Definition \ref{d5.2}). Thus the proof of Theorem
\ref{t1.4} is finished.

After a preprint of this article was posted in arXiv, Shizuo Endo
kindly informed us that he had another proof of Theoem \ref{t1.4}
by applying Theoem 3.3 of his joint paper with Miyata \cite{EM},
which is included in the appendix of this paper.

\bigskip
This paper is organized as follows. Section 2 contains the
preliminaries of $G$-lattices and the flabby class monoid of
$G$-lattices. In Section 3 we construct six $D_n$-lattices where
$n$ is an odd integer. Then we show that the rank $n+1$ lattices
are stably permutation. When $n=p$ is an odd prime number, these
lattices play an important role in Section 4. Section 4 begins
with the classification of indecomposable $D_p$-lattices due to
Myrna Pike Lee in \cite[page 752]{Le,CR}. Then we prove that, if
$h_p^+=1$, all the flabby $D_p$-lattices are stably permutation.
This finishes the proof of one direction of Theorem \ref{t1.4}. In
Section 5 we recall Diederichsen-Reiner's Theorem of integral
representations of $C_p$ and their invariants \cite[page
729]{Di,Re,CR}. Then we show that, the flabby lattices constructed
in Section 3 are the only indecomposable flabby $D_p$-lattices
which are stably permutation. Thus the proof of another direction
of Theorem \ref{t1.4} is finished. In the last section Theorem
\ref{t1.5} is proved.

\medskip
Terminology and notations.
In this paper, $C_n$ and $D_n$ denote the cyclic group of order $n$ and the dihedral group of order $2n$ respectively.
$\zeta_n$ denotes a primitive $n$-th root of unity.
$h_n$ and $h_n^+$ denote the class numbers of the $n$-th cyclotomic field $\bm{Q}(\zeta_n)$ and its real subfield $\bm{Q}(\zeta_n+\zeta_n^{-1})$ respectively.

If $R$ is a Dedekind domain, recall the definition of the (ideal)
class group $C(R)$. Let $Div(R)$ denote the set of all non-zero
fraction ideals of $R$, which is a group under the multiplication
of ideals; let $Prin(R)$ be the subgroup of $Div(R)$ consisting of
principal ideals. The class group $C(R)$ is the quotient group
$Div(R)/Prin(R)$; if $I$ is a fractional ideal, $[I]$ denotes the
image of $I$ in $C(R)$.

\section{\boldmath Preliminaries of $G$-lattices}

Throughout this section, $k$ is a field.

Let $\Gamma_k=\fn{Gal}(k_{\fn{sep}}/k)$. A $\Gamma_k$-lattice $M$
is a free abelian group of finite rank on which $\Gamma_k$ acts
continuously. It is known that the category of algebraic tori
defined over $k$ is anti-equivalent to the category of
$\Gamma_k$-lattices \cite[page 27; Sw2, page 36]{Vo}. If $T$ is a
torus, the $\Gamma_k$-lattice corresponding to $T$ is its
character module $M:=\fn{Hom}(T,\bm{G}_{m,k_{\fn{sep}}})$. Let
$\Gamma_0$ be an open subgroup of $\Gamma_k$ such that $\Gamma_0$
acts trivially on $M$. Consider $M$ as a $G$-lattice where
$G=\Gamma_k/\Gamma_0$ is a finite group. Thus we are led to the
following formulation.

Let $G$ be a finite group. Recall that a finitely generated
$\bm{Z}[G]$-module $M$ is called a $G$-lattice if it is
torsion-free as an abelian group. We define
$\fn{rank}_{\bm{Z}}M=n$ if $M$ is a free abelian group of
$\fn{rank} n$.

A $G$-lattice $M$ is called a permutation lattice if $M$ has a $\bm{Z}$-basis permuted by $G$.
A $G$-lattice $M$ is called stably permutation if $M\oplus P$ is a permutation lattice where $P$ is some permutation lattice.
$M$ is called an invertible lattice if it is a direct summand of some permutation lattice.
A $G$-lattice $M$ is called a flabby lattice if $H^{-1}(S,M)=0$ for any subgroup $S$ of $G$;
it is called coflabby if $H^1(S,M)=0$ for any subgroup $S$ of $G$.
For details, see \cite{CTS,Sw2,Lo}.

\medskip
Let $G$ be a finite group.
Two $G$-lattices $M_1$ and $M_2$ are similar,
denoted by $M_1\sim M_2$, if $M_1\oplus P_1 \simeq M_2 \oplus P_2$ for some permutation $G$-lattices $P_1$ and $P_2$.
The flabby class monoid $F_G$ is the class of all flabby $G$-lattices under the similarity relation.
In particular, if $M$ is a flabby lattice,
$[M]\in F_G$ denotes the equivalence class containing $M$;
we define $[M_1]+[M_2]=[M_1\oplus M_2]$ and thus $F_G$ becomes an abelian monoid \cite{Sw2}.

\begin{defn} \label{d2.3}
Let $G$ be a finite group, $M$ be any $G$-lattice. Then $M$ has a
flabby resolution, i.e. there is an exact sequence of
$G$-lattices: $0\to M\to P\to E\to 0$ where $P$ is a permutation
lattice and $E$ is a flabby lattice. The class $[E]\in F_G$ is
uniquely determined by the lattice $M$ \cite{Sw2}. We define
$[M]^{fl}=[E]\in F_G$, following the nomenclature in \cite[page
38]{Lo}. Sometimes we will say that $[M]^{fl}$ is permutation or
invertible if the class $[E]$ contains a permutation or invertible
lattice.
\end{defn}

\begin{defn} \label{d2.4}
Let $K/k$ be a finite Galois field extension with $G=\fn{Gal}(K/k)$.
Let $M=\bigoplus_{1\le i\le n} \bm{Z}\cdot e_i$ be a $G$-lattice.
We define an action of $G$ on $K(M)=K(x_1,\ldots,x_n)$,
the rational function field of $n$ variables over $K$,
by $\sigma\cdot x_j=\prod_{1\le i\le n} x_i^{a_{ij}}$ if $\sigma\cdot e_j=\sum_{1\le i\le n} a_{ij} e_i \in M$,
for any $\sigma \in G$ (note that $G$ acts on $K$ also).
The fixed field is denoted by $K(M)^G$.
\end{defn}

If $T$ is an algebraic torus over $k$ satisfying
$T\times_{\fn{Spec}(k)}\fn{Spec}(K)\simeq\bm{G}_{m,K}^n$ where
$\bm{G}_{m,K}$ is the one-dimensional multiplicative group over
$K$, then $M:=\fn{Hom}(T,\bm{G}_{m,K})$ is a $G$-lattice and the
function field of $T$ over $k$ is isomorphic to $K(M)^G$ by Galois
descent \cite[page 36]{Sw2}. Thus the stable rationality of $T$
over $k$ is equivalent to that of $K(M)^G$. Such a torus $T$ is
called a $G$-torus over $k$ (see Definition \ref{d1.9}).

\begin{defn} \label{d2.5}

We give a generalization of $K(M)^G$. Let $M=\bigoplus_{1\le i\le
n} \bm{Z}\cdot e_i$ be a $G$-lattice, $k'/k$ be a finite Galois
extension field such that there is a surjection $G\to
\fn{Gal}(k'/k)$. Thus $G$ acts naturally on $k'$ by
$k$-automorphisms. We define an action of $G$ on
$k'(M)=k'(x_1,\ldots,x_n)$ in a similar way as $K(M)$. The fixed
field is denoted by $k'(M)^G$. The action of $G$ on $k'(M)$ is
called a purely quasi-monomial action in \cite[Definition
1.1]{HKK}; it is possible that $G$ acts faithfully on $k'$ (the
case $k'=K$) or trivially on $k'$ (the case $k'=k$) .

\end{defn}

\begin{theorem} \label{t2.6}
Let $K/k$ be a finite Galois extension field, $G=\fn{Gal}(K/k)$
and $M$ be a $G$-lattice.
\begin{enumerate}
\item[{\rm (1)}] {\rm (Voskresenskii, Endo and Miyata
\cite[Theorem 1.2; Len, Theorem 1.7]{EM1})} $K(M)^G$ is stably
$k$-rational if and only if $[M]^{fl}$ is permutation, i.e.\ there
exists a short exact sequence of $G$-lattices $0\to M\to P_1\to
P_2\to 0$ where $P_1$ and $P_2$ are permutation $G$-lattices.
\item[{\rm (2)}] {\rm (Saltman \cite[Theorem 3.14; Ka2, Theorem
2.8]{Sa})} $K(M)^G$ is retract $k$-rational if and only if
$[M]^{fl}$ is invertible.
\end{enumerate}
\end{theorem}

\begin{theorem}[{Endo and Miyata \cite[Theorem 1.5; Sw3, Theorem 3.4; Lo, 2.10.1]{EM}}] \label{t2.7}
Let $G$ be a finite group. Then all the flabby $G$-lattices are
invertible if and only if all the Sylow subgroups of $G$ are
cyclic.
\end{theorem}

\section{\boldmath Some $D_n$-lattices}

Throughout this section, $G$ denotes the group
$G=\langle\sigma,\tau:\sigma^n=\tau^2=1,\tau\sigma\tau^{-1}=\sigma^{-1}\rangle$
where $n \ge 3$ is an odd integer, i.e.\ $G$ is the dihedral group
$D_n$. Define $H=\langle \tau \rangle$.

We will construct six $G$-lattices which will become
indecomposable $G$-lattices if $n=p$ is an odd prime number (to be
proved in Section 4).

\begin{defn} \label{d3.1}
Let $G=\langle\sigma,\tau\rangle$ be the dihedral group defined
before. Define $G$-lattices $M_+$ and $M_-$ by $M_+=\fn{Ind}^G_H
\bm{Z}$, $M_-=\fn{Ind}^G_H \bm{Z}_-$ the induced lattices where
$\bm{Z}$ and $\bm{Z}_-$ are $H$-lattices such that $\tau$ acts on
$\bm{Z}= \bm{Z}\cdot u$, $\bm{Z}_-=\bm{Z}\cdot u'$ by $\tau\cdot
u=u$, $\tau\cdot u'=-u'$ respectively (note that $u$ and $u'$ are
the generators of $\bm{Z}$ and $\bm{Z}_-$ as abelian groups). By
choosing a $\bm{Z}$-basis for $M_+$ corresponding to $\sigma^i
u\in \fn{Ind}^G_H\bm{Z}$ (where $0\le i\le n-1$), the actions of
$\sigma$ and $\tau$ on $M_+$ are given by the $n\times n$ integral
matrices
\[
\sigma\mapsto A=\begin{pmatrix}
0 & 0 & \cdots & 0 & 1 \\
1 & 0 & & & 0 \\
& \ddots & \ddots & & \vdots \\
& & 1 & 0 & 0 \\
& & & 1 & 0
\end{pmatrix}, \quad
\tau\mapsto B=\left(\begin{array}{@{}cccc;{3pt/2pt}c@{}}
& & & 1 & \\ & & \iddots & & \\ & 1 & & & \\ 1 & & & & \\ \hdashline[3pt/2pt] & & & & 1
\end{array}\right).
\]

Similarly, for a $\bm{Z}$-basis for $M_-$ corresponding to $\sigma^i u'$,
the actions of $\sigma$ and $\tau$ are given by
\[
\sigma\mapsto A=\begin{pmatrix}
0 & 0 & \cdots & 0 & 1 \\
1 & 0 & & & 0 \\
& \ddots & \ddots & & \vdots \\
& & 1 & 0 & 0 \\
& & & 1 & 0
\end{pmatrix}, \quad
\tau\mapsto -B=\left(\begin{array}{@{}cccc;{3pt/2pt}c@{}}
& & & -1 & \\ & & \iddots & & \\ & -1 & & & \\ -1 & & & & \\ \hdashline[3pt/2pt] & & & & -1
\end{array}\right).
\]
\end{defn}

\begin{defn} \label{d3.2}
As before, $G=\langle \sigma,\tau\rangle \simeq D_n$. Let
$f(\sigma)=1+\sigma+\cdots+\sigma^{n-1}\in \bm{Z}[G]$. Since
$\tau\cdot f(\sigma)=f(\sigma)\cdot \tau$, the left
$\bm{Z}[G]$-ideal $\bm{Z}[G]\cdot f(\sigma)$ is a two-sided ideal;
as an ideal in $\bm{Z}[G]$, we denote it by $\langle
f(\sigma)\rangle$. The natural projection $\bm{Z}[G]\to
\bm{Z}[G]/\langle f(\sigma)\rangle$ induces an isomorphism of
$\bm{Z}[G]/\langle f(\sigma)\rangle$ and the twisted group ring
$\bm{Z}[\zeta_n]\circ H$ (see \cite[p.589]{CR}). Explicitly, let
$\zeta_n$ be a primitive $n$-th root of unity. Then
$\bm{Z}[\zeta_n]\circ H=\bm{Z}[\zeta_n]\oplus \bm{Z}[\zeta_n]\cdot
\tau$ and $\tau\cdot \zeta_n=\zeta_n^{-1}$. If $n=p$ is an odd
prime number, $\bm{Z}[\zeta_p]\circ H$ is a hereditary order
\cite[pages 593--595]{CR}. Note that we have the following fibre
product diagram
\[
\xymatrix{\bm{Z}[G] \ar[r] \ar[d] & \bm{Z}[\zeta_n]\circ H \ar[d] \\
\bm{Z}[H] \ar[r] & \overline{\bm{Z}}[H]}
\]
where $\overline{\bm{Z}}=\bm{Z}/n\bm{Z}$ (compare with \cite[page 748, (34.43)]{CR}).
\end{defn}

Using the $G$-lattices $M_+$ and $M_-$ in Definition \ref{d3.1},
define $N_+=\bm{Z}[G]/\langle f(\sigma)\rangle \otimes_{\bm{Z}[G]} M_+ =M_+/f(\sigma)M_+$,
$N_-=\bm{Z}[G]/\langle f(\sigma)\rangle \otimes_{\bm{Z}[G]} M_-=M_-/f(\sigma) M_-$.

The $\bm{Z}[G]/\langle f(\sigma)\rangle$-lattices $N_+$ and $N_-$ may be regarded as $G$-lattices
through the $\bm{Z}$-algebra morphism $\bm{Z}[G]\to \bm{Z}[G]/\langle f(\sigma)\rangle$.
By choosing a $\bm{Z}$-basis for $N_+$ corresponding to $\sigma^i u$ where $1\le i\le n-1$,
the actions of $\sigma$ and $\tau$ on $N_+$ are given by the $(n-1)\times (n-1)$ integral matrices
\[
\sigma\mapsto A'=\begin{pmatrix}
0 & 0 & 0 & \cdots & 0 & -1 \\
1 & 0 & & & & -1 \\
& 1 & 0 & & & -1 \\
& & & \ddots & & \vdots \\
& & & & 0 & -1 \\
& & & & 1 & -1
\end{pmatrix}, \quad
\tau\mapsto B'=\begin{pmatrix}
& & & 1 \\ & & \iddots & \\ & 1 & & \\ 1 & & &
\end{pmatrix}.
\]

Similarly, the actions of $\sigma$ and $\tau$ on $N_-$ are given by
\[
\sigma\mapsto A'=\begin{pmatrix}
0 & 0 & 0 & \cdots & 0 & -1 \\
1 & 0 & & & & -1 \\
& 1 & 0 & & & -1 \\
& & & \ddots & & \vdots \\
& & & & 0 & -1 \\
& & & & 1 & -1
\end{pmatrix}, \quad
\tau\mapsto -B'=\begin{pmatrix}
& & & -1 \\ & & \iddots & \\ & -1 & & \\ -1 & & &
\end{pmatrix}.
\]

\begin{defn} \label{d3.3}
We will use the $G$-lattices $M_+$ and $M_-$ in Definition
\ref{d3.1} to construct $G$-lattices $\widetilde{M}_+$ and
$\widetilde{M}_-$ which are of rank $n+1$ satisfying the short
exact sequences of $G$-lattices
\begin{align*}
0 &\to M_+\to \widetilde{M}_+\to \bm{Z}\to 0 \\
0 &\to M_-\to \widetilde{M}_-\to \bm{Z}_-\to 0
\end{align*}
where the $\bm{Z}$-lattice structures of $\widetilde{M}_+$ and $\widetilde{M}_-$ will be described below and $\bm{Z}=\bm{Z}\cdot w$,
$\bm{Z}_-=\bm{Z}\cdot w'$ are $G$-lattices defined by $\sigma\cdot w=w$, $\tau\cdot w=w$, $\sigma\cdot w'=w'$, $\tau\cdot w'=-w'$.
\end{defn}

Let $\{w_i:0\le i\le n-1\}$ be the $\bm{Z}$-basis of $M_+$ in Definition \ref{d3.1}.
As a free abelian group, $\widetilde{M}_+=(\bigoplus_{0\le i\le n-1} \bm{Z}\cdot w_i) \oplus \bm{Z}\cdot w$.
Define the actions of $\sigma$ and $\tau$ on $\widetilde{M}_+$ by the $(n+1)\times (n+1)$ integral matrices
\[
\sigma\mapsto \tilde{A}=\left(\begin{array}{@{\quad}c@{\quad~};{3pt/2pt}c@{}}
 & \\ A & \\  & \\ \hdashline[3pt/2pt] & 1
\end{array}\right),\quad
\tau\mapsto \tilde{B}=\left(\begin{array}{@{\quad}c@{\quad~};{3pt/2pt}c@{}}
 & 1 \\ B & \vdots \\ & 1 \\ \hdashline[3pt/2pt] & -1
\end{array}\right).
\]

Similarly, let $\{w_i:0\le i\le n-1\}$ be the $\bm{Z}$-basis of $M_-$ in Definition \ref{d3.1},
and $\widetilde{M}_-=(\bigoplus_{0\le i\le n-1} \bm{Z} w_i)\oplus\bm{Z}\cdot w'$.
Define the actions of $\sigma$ and $\tau$ on $\widetilde{M}_-$ by
\[
\sigma\mapsto \tilde{A}=\left(\begin{array}{@{\quad}c@{\quad~};{3pt/2pt}c@{}}
 & \\ A & \\  & \\ \hdashline[3pt/2pt] & 1
\end{array}\right),\quad
\tau\mapsto -\tilde{B}=\left(\begin{array}{@{\quad}c@{\quad~};{3pt/2pt}c@{}}
 & -1 \\ -B & \vdots \\ & -1 \\ \hdashline[3pt/2pt] & 1
\end{array}\right).
\]

In the remaining part of this section, we will show that
$\widetilde{M}_+$ and $\widetilde{M}_-$ are stably permutation
$G$-lattices, and $\widetilde{M}_+ \oplus \widetilde{M}_-$ is a
permutation $G$-lattice.

\begin{theorem} \label{t3.4}
Let $G=\langle\sigma,\tau:\sigma^n=\tau^2=1,\tau\sigma\tau^{-1}=\sigma^{-1}\rangle \simeq D_n$ where $n$ is an odd integer.
Then $\widetilde{M}_+\oplus \bm{Z}\simeq \bm{Z}[G/\langle\sigma\rangle]\oplus\bm{Z}[G/\langle\tau\rangle]$.
\end{theorem}

\begin{proof}
Let $u_0$, $u_1$ be the $\bm{Z}$-basis of
$\bm{Z}[G/\langle\sigma\rangle]$ correspond to 1, $\tau$. Then
$\sigma:u_0\mapsto u_0$, $u_1\mapsto u_1$,
$\tau:u_0\leftrightarrow u_1$.

Let $\{v_i:0\le i\le n-1\}$ be the $\bm{Z}$-basis of
$\bm{Z}[G/\langle\tau\rangle]$ correspond to $\sigma^i$ where
$0\le i\le n-1$. Then $\sigma: v_i\mapsto v_{i+1}$ (where the
index is understood modulo $n$), $\tau: v_i\mapsto v_{n-i}$ for
$0\le i\le n-1$.

It follows that $u_0,u_1,v_0,v_1,\ldots,v_{n-1}$ is a $\bm{Z}$-basis of $\bm{Z}[G/\langle\sigma\rangle]\oplus\bm{Z}[G/\langle\tau\rangle]$.

Define
\begin{gather*}
t = u_0+u_1+\sum_{0\le i\le n-1} v_i, \\
x = u_0+u_1+\sum_{1\le i\le n-1} v_i,\quad y= \frac{n-1}{2}u_0+\frac{n+1}{2}u_1+\frac{n-1}{2} \sum_{0\le i\le n-1}v_i.
\end{gather*}

Since $\sum_{0\le i\le n} \sigma^i\cdot (v_1+\cdots+v_{n-1})=(n-1)\sum_{0\le i\le n-1} v_i$,
it follows that $\tau\cdot y=-y+\sum_{0\le i\le n-1} \sigma^i (x)$.
Then it is routine to verify that
\[
\left(\bigoplus_{0\le i\le n-1} \bm{Z}\cdot \sigma^i(x)\right) \oplus \bm{Z}\cdot y \simeq \widetilde{M}_+, \quad \bm{Z}\cdot t\simeq \bm{Z}
\]
by checking the actions of $\sigma$ and $\tau$ on lattices in both sides.

Now we will show that $\sigma(x),
\sigma^2(x),\ldots,\sigma^{n-1}(x),x,y,t$ is a $\bm{Z}$-basis of
$\bm{Z}[G/\langle\sigma\rangle]\oplus\bm{Z}[G/\langle\tau\rangle]$.
Write the determinant of these $n+2$ elements with respect to the
$\bm{Z}$-basis $u_0,u_1,v_0,v_1,\ldots,v_{n-1}$. We get the
coefficient matrix $T$ as
\begin{gather*}
\renewcommand{\arraystretch}{1.15} \normalsize
T=\left(\begin{array}{@{\,}cccc;{3pt/2pt}cc@{\,}}
1 & 1 & \cdots & 1 & \frac{n-1}{2} & 1 \\
1 & 1 & \cdots & 1 & \frac{n+1}{2} & 1 \\[1mm] \hdashline[3pt/2pt]
1 & 1 & & 0 & \frac{n-1}{2} & 1 \\
0 & 1 & & 1 & \frac{n-1}{2} & 1 \\
1 & 0 & & 1 & \frac{n-1}{2} & 1 \\
\vdots & \vdots & & \vdots & \vdots & \vdots \\
1 & 1 & & 1 & \frac{n-1}{2} & 1
\end{array}\right) \hspace*{-44mm}
\begin{minipage}[c][0mm]{56mm}
$\begin{blockarray}{cl}
~ & ~ \\[19mm]
\begin{block}{c\}l}
~ & \\ & \\ & \text{$n$ rows} \\[2pt] & \\ & \\
\end{block}
~ & ~ \\[-10mm] \underbrace{\hspace*{20mm}}_{\text{\normalsize$n$ columns}}\hspace*{21mm} &
\end{blockarray}$
\end{minipage} \\[1mm]
\end{gather*}

The determinant of $T$ may be calculated as follows: Subtract the
last column from each of the first $n$ columns. Also subtract
$\frac{n-1}{2}$ times of the last column from the $(n+1)$-th
column. Then it is easy to see $\det (T)= 1$.
\end{proof}

\begin{theorem} \label{t3.5}
Let $G=\langle\sigma,\tau:\sigma^n=\tau^2=1,\tau\sigma\tau^{-1}=\sigma^{-1}\rangle\simeq D_n$ where $n$ is an odd integer.
Then $\widetilde{M}_-\oplus\bm{Z}[G/\langle\tau\rangle]\simeq \bm{Z}[G]\oplus\bm{Z}$.
\end{theorem}

\begin{proof}
The idea of the proof is similar to that of Theorem \ref{t3.4}.

Let $u_0,u_1,\ldots,u_{n-1},v_0,v_1,\ldots,v_{n-1},t$ be a $\bm{Z}$-basis of $\bm{Z}[G]\oplus\bm{Z}$ where $u_i$,
$v_j$ correspond to $\sigma^i$, $\sigma^j\tau$ in $\bm{Z}[G]$.
The actions of $\sigma$ and $\tau$ are given by
\begin{align*}
\sigma &: u_i\mapsto u_{i+1},~ v_j\mapsto v_{j+1},~ t\mapsto t, \\
\tau &: u_i\leftrightarrow v_{n-i},~ t\mapsto t
\end{align*}
where the index of $u_i$ or $v_j$ is understood modulo $n$.

Define $x,y,z\in \bm{Z}[G]\oplus \bm{Z}$ by
\begin{gather*}
x = u_0-v_0, \quad y=\left(\sum_{0\le i\le n-1} u_i\right)+t, \\
z =\left(\sum_{1\le i\le \frac{n-1}{2}} u_i\right)+\left(\sum_{\frac{n+1}{2}\le j\le n-1} v_j\right)+t.
\end{gather*}

We claim that
\[
\left(\bigoplus_{0\le i\le n-1} \bm{Z}\cdot \sigma^i(x)\right)\oplus \bm{Z}\cdot y\simeq \widetilde{M}_-, \quad
\bigoplus_{0\le i\le n-1} \bm{Z}\cdot \sigma^i(z)\simeq \bm{Z}[G/\langle\tau\rangle].
\]

Since $\tau(x)=-x$, $\tau(z)=z$,
it follows that $\tau\cdot \sigma^i(x)=-\sigma^{n-i}(x)$, $\tau\cdot\sigma^i(z)=\sigma^{n-i}(z)$.
The remaining proof is omitted.

Now we will show that
$\sigma(x),\sigma^2(x),\ldots,\sigma^{n-1}(x),
x,y,\sigma(z),\ldots,\sigma^{n-1}(z)$, $z$ form a $\bm{Z}$-basis
of $\bm{Z}[G]\oplus \bm{Z}$. Write the coefficient matrix of these
elements with respect to the $\bm{Z}$-basis
$u_0,u_1,\ldots,u_{n-1},v_0,v_1,\ldots,v_{n-1},t$. We get
$\det(T_n)$ where $T_n$ is a $(2n+1)\times (2n+1)$ integral
matrix. For example,
\[
T_3=\left(\begin{array}{@{}ccc;{3pt/2pt}c;{3pt/2pt}ccc@{}}
0 & 0 & 1 & 1 & 0 & 0 & 1 \\
1 & 0 & 0 & 1 & 1 & 0 & 0 \\
0 & 1 & 0 & 1 & 0 & 1 & 0 \\ \hdashline[3pt/2pt]
0 & 0 & -1 & 0 & 0 & 1 & 0 \\
-1 & 0 & 0 & 0 & 0 & 0 & 1 \\
0 & -1 & 0 & 0 & 1 & 0 & 0 \\ \hdashline[3pt/2pt]
0 & 0 & 0 & 1 & 1 & 1 & 1
\end{array}\right).
\]

For any $n$, we evaluate $\det (T_n)$ by adding the $i$-th row to
$(i+n)$-th row of $T_n$ for $1\le i\le n$. We find that
$\det(T_n)=\pm\det(T')$ where $T'$ is an $(n+1)\times (n+1)$
integral matrix. Note that all the entries of the $i$-th row of
$T'$ (where $1\le i\le n$) are one except one position, because of
the definition of $z$ (and those of $\sigma^i(z)$ for $1\le i\le
n-1$). Subtract the last row from the $i$-th row where $1\le i\le
n$. We find $\det(T')=\pm 1$.
\end{proof}


Before proving Theorem \ref{t3.6}, we define the following matrix
first. Let
\begin{align*}
{\rm Circ}(c_0,c_1,\ldots,c_{n-1})=
\left(
\begin{array}{@{}ccccc@{}}
 c_0 & c_{n-1} & \cdots & c_2 & c_1 \\
 c_1 & c_0 & c_{n-1} &  & c_2 \\
 \vdots & c_1 & c_0 & \ddots & \vdots \\
 c_{n-2} &  & \ddots & \ddots & c_{n-1} \\
 c_{n-1} & c_{n-2} & \cdots & c_1 & c_0
\end{array}
\right)
\end{align*}
be the $n\times n$ circulant matrix whose determinant is
\begin{align*}
\det({\rm Circ}(c_0,c_1,\ldots,c_{n-1}))
&=\prod_{k=0}^{n-1}(c_0+c_1\zeta_n^k+\cdots+c_{n-1}\zeta_n^{(n-1)k}).
\end{align*}

\begin{lemma}\label{lemcirc}
Let $n\geq 3$ be an odd integer.\\
{\rm (1)}
$\displaystyle{\det({\rm Circ}(\overbrace{1,\ldots,1}^\frac{n-1}{2},
\overbrace{0,\ldots,0}^\frac{n+1}{2}))=\frac{n-1}{2}}$. \\
{\rm (2)}
$\det({\rm Circ}(\overbrace{-1,\ldots,-1}^\frac{n-1}{2},0,\overbrace{1,\ldots,1}^\frac{n-3}{2},0))=-1$.
\end{lemma}

\begin{proof}
(1) follows from
\begin{align*}
\det({\rm Circ}(\overbrace{1,\ldots,1}^\frac{n-1}{2},
\overbrace{0,\ldots,0}^\frac{n+1}{2}))
&=\prod_{k=0}^{n-1}\left(1+\zeta_n^k+\cdots+\zeta_n^{\frac{n-3}{2}k}\right)\\
&=\frac{n-1}{2} \prod_{k=1}^{n-1}\left(1+\zeta_n^k+\cdots+\zeta_n^{\frac{n-3}{2}k}\right)\\
&=\frac{n-1}{2}
\end{align*}
because $1+\zeta_n+\cdots+\zeta_n^{\frac{n-3}{2}}=\frac{1-\zeta_n^{\frac{n-1}{2}}}{1-\zeta_n}$ is a cyclotomic unit with
\[
\left(\frac{1-\zeta_n^{\frac{n-1}{2}}}{1-\zeta_n}\right)^{-1}
=\frac{1-\zeta_n}{1-\zeta_n^{\frac{n-1}{2}}}
=\frac{1-\zeta_n}{1-\zeta_n^{n-1}}(1+\zeta_n^{\frac{n-1}{2}})
=-\zeta_n(1+\zeta_n^{\frac{n-1}{2}}).
\]
(2) follows from
\begin{align*}
&\det({\rm Circ}(\overbrace{-1,\ldots,-1}^\frac{n-1}{2},0,\overbrace{1,\ldots,1}^\frac{n-3}{2},0))\\
&=\prod_{k=0}^{n-1}\left(-1-\zeta_n^k-\cdots-\zeta_n^{\frac{n-3}{2}k}+\zeta_n^{\frac{n+1}{2}k}+\cdots+\zeta_n^{(n-2)k}\right)\\
&=(-1)\prod_{k=1}^{n-1}\left(-1-\zeta_n^k-\cdots-\zeta_n^{\frac{n-3}{2}k}+\zeta_n^{\frac{n+1}{2}k}+\cdots+\zeta_n^{(n-2)k}\right)\\
&=-1
\end{align*}
because
$-1-\zeta_n-\cdots-\zeta_n^\frac{n-3}{2}+\zeta_n^\frac{n+1}{2}+\cdots+\zeta_n^{n-2}$
is a unit with
\begin{align*}
&\big(-1-\zeta_n-\cdots-\zeta_n^\frac{n-3}{2}+\zeta_n^\frac{n+1}{2}+\cdots+\zeta_n^{n-2}\big)^{-1}\\
&=\begin{cases}
\sum_{k=1}^\frac{n+1}{2} (-1)^k \zeta_n^k & n \equiv 1 \pmod{4} \\
-\sum_{k=0}^\frac{n-3}{2} (-1)^k \zeta_n^{-k} & n \equiv 3 \pmod {4}.
\end{cases}
\end{align*}
\end{proof}

\begin{theorem} \label{t3.6}
Let $G=\langle\sigma,\tau: \sigma^n=\tau^2=1,\tau\sigma\tau^{-1}=\sigma^{-1}\rangle\simeq D_n$ where $n$ is an odd integer.
Then $\widetilde{M}_+\oplus \widetilde{M}_-\simeq \bm{Z}[G]\oplus \bm{Z}[G/\langle \sigma \rangle]$.
\end{theorem}

\begin{proof}
Let $u_0,u_1,\ldots,u_{n-1},v_0,v_1,\ldots,v_{n-1}$ be
a $\bm{Z}$-basis of $\bm{Z}[G]$ where $u_i$,
$v_j$ correspond to $\sigma^i$, $\sigma^j\tau$ in $\bm{Z}[G]$.
The actions of $\sigma$ and $\tau$ are given by
\begin{align*}
\sigma &: u_i\mapsto u_{i+1},~ v_j\mapsto v_{j+1},\\
\tau &: u_i\leftrightarrow v_{n-i}
\end{align*}
where the index of $u_i$ or $v_j$ is understood modulo $n$. Let
$t_0$, $t_1$ be the $\bm{Z}$-basis of
$\bm{Z}[G]/\langle\sigma\rangle]$ correspond to 1, $\tau$. Then
$\sigma:t_0\mapsto t_0$, $t_1\mapsto t_1$,
$\tau:t_0\leftrightarrow t_1$.

Define $x,y_0,z,y_1\in \bm{Z}[G]\oplus \bm{Z}[G/\langle\sigma\rangle]$ by
\begin{align*}
x &=u_0+\left(\sum_{\frac{n+3}{2}\leq i\leq n-1}u_i\right)
+\left(\sum_{2\leq j\leq \frac{n+1}{2}}v_j\right)+t_0+t_1,\\
y_0&=\frac{n-1}{2}\left(\sum_{0\le j\le n-1} v_j\right)+t_0+(n-1)t_1,\\
z &=u_0+u_1+\left(\sum_{\frac{n+3}{2}\le i\le n-1} u_i\right)
-\left(\sum_{1 \le j\le \frac{n+1}{2}} v_j\right)+t_0-t_1,\\
y_1&=\left(\sum_{0\leq i\leq n-1}u_i\right)
-\frac{n-1}{2}\left(\sum_{0\le j\le n-1} v_j\right)+t_0-(n-1)t_1.
\end{align*}

It is easy to verify that
\[
\left(\bigoplus_{0\le i\le n-1} \bm{Z}\cdot \sigma^i(x)\right)\oplus \bm{Z}\cdot y_0\simeq \widetilde{M}_+, \quad
\left(\bigoplus_{0\le i\le n-1} \bm{Z}\cdot \sigma^i(z)\right)\oplus \bm{Z}\cdot y_1\simeq \widetilde{M}_-.
\]

It remains to show that $\sigma^{\frac{n-3}{2}}(x)$, $\ldots$,
$\sigma^{n-1}(x)$, $x$, $\sigma(x)$, $\ldots$,
$\sigma^{\frac{n-5}{2}}(x)$, $y_0$, $\sigma^{\frac{n-3}{2}}(z)$,
$\ldots$, $\sigma^{n-1}(z)$, $z$, $\sigma(z)$, $\ldots$,
$\sigma^{\frac{n-5}{2}}(z)$, $y_1$ form a $\bm{Z}$-basis of
$\bm{Z}[G]\oplus \bm{Z}[G/\langle\sigma\rangle]$. Let $Q$ be the
coefficient matrix of $\sigma^{\frac{n-3}{2}}(x)$, $\ldots$,
$\sigma^{n-1}(x)$, $x$, $\sigma(x)$, $\ldots$,
$\sigma^{\frac{n-5}{2}}(x)$, $y_0$, $\sigma^{\frac{n-3}{2}}(z)$,
$\ldots$, $\sigma^{n-1}(z)$, $z$, $\sigma(z)$, $\ldots$,
$\sigma^{\frac{n-5}{2}}(z)$, $y_1$ with respect to the
$\bm{Z}$-basis $u_0, u_1,\ldots,u_{n-1},v_0, v_1,\ldots,v_{n-1},
t_0,t_1$.

The matrix $Q$ is defined as
\begin{align*}
Q=
\left(\begin{array}{@{}ccc;{3pt/2pt}c;{3pt/2pt}ccc;{3pt/2pt}c@{}}
 & & & 0 & & & & 1\\
\multicolumn{3}{c;{3pt/2pt}}{{\rm Circ}(\overbrace{1,\ldots,1}^\frac{n-1}{2},
\overbrace{0,\ldots,0}^\frac{n+1}{2})} & \vdots &
\multicolumn{3}{c;{3pt/2pt}}{{\rm Circ}(\overbrace{1,\ldots,1}^\frac{n+1}{2},
\overbrace{0,\ldots,0}^\frac{n-1}{2})} & \vdots\\
 & & & 0 & & & & 1\\\hdashline[3pt/2pt]
 & & & \frac{n-1}{2} & & & & -\frac{n-1}{2}\\
\multicolumn{3}{c;{3pt/2pt}}{{\rm Circ}(\overbrace{0,\ldots,0}^\frac{n+1}{2},
\overbrace{1,\ldots,1}^\frac{n-1}{2})} & \vdots &
\multicolumn{3}{c;{3pt/2pt}}{{\rm Circ}(\overbrace{0,\ldots,0}^\frac{n-1}{2},
\overbrace{-1,\ldots,-1}^\frac{n+1}{2})} & \vdots\\
 & & & \frac{n-1}{2} & & & & -\frac{n-1}{2}\\\hdashline[3pt/2pt]
\multicolumn{3}{c;{3pt/2pt}}{1\qquad\quad \cdots\qquad\quad 1} & 1 & \multicolumn{3}{c;{3pt/2pt}}{1\qquad\quad \cdots\qquad\quad 1} & 1\\\hdashline[3pt/2pt]
\multicolumn{3}{c;{3pt/2pt}}{\underbrace{1\qquad\quad \cdots\qquad\quad 1}_{n}} & n-1 & \multicolumn{3}{c;{3pt/2pt}}{\underbrace{-1\qquad\ \ \cdots\qquad\ \ -1}_{n}}  & -(n-1)
\end{array}\right)
\begin{array}{c}
\Bigg\}\ n \\ \\ \\ \Bigg\}\ n \\ \\ \\ \\
\end{array}.
\end{align*}

For examples, when $n=3,5$, $Q$ is of the form
\[
\left(
\begin{array}{@{}ccc;{3pt/2pt}c;{3pt/2pt}ccc;{3pt/2pt}c@{}}
 1 & 0 & 0 & 0 & 1 & 0 & 1 & 1 \\
 0 & 1 & 0 & 0 & 1 & 1 & 0 & 1 \\
 0 & 0 & 1 & 0 & 0 & 1 & 1 & 1 \\\hdashline[3pt/2pt]
 0 & 1 & 0 & 1 & 0 & -1 & -1 & -1 \\
 0 & 0 & 1 & 1 & -1 & 0 & -1 & -1 \\
 1 & 0 & 0 & 1 & -1 & -1 & 0 & -1 \\\hdashline[3pt/2pt]
 1 & 1 & 1 & 1 & 1 & 1 & 1 & 1 \\\hdashline[3pt/2pt]
 1 & 1 & 1 & 2 & -1 & -1 & -1 & -2 \\
\end{array}
\right)
\]
and
\[
\left(
\begin{array}{@{}ccccc;{3pt/2pt}c;{3pt/2pt}ccccc;{3pt/2pt}c@{}}
 1 & 0 & 0 & 0 & 1 & 0 & 1 & 0 & 0 & 1 & 1 & 1 \\
 1 & 1 & 0 & 0 & 0 & 0 & 1 & 1 & 0 & 0 & 1 & 1 \\
 0 & 1 & 1 & 0 & 0 & 0 & 1 & 1 & 1 & 0 & 0 & 1 \\
 0 & 0 & 1 & 1 & 0 & 0 & 0 & 1 & 1 & 1 & 0 & 1 \\
 0 & 0 & 0 & 1 & 1 & 0 & 0 & 0 & 1 & 1 & 1 & 1 \\\hdashline[3pt/2pt]
 0 & 1 & 1 & 0 & 0 & 2 & 0 & -1 & -1 & -1 & 0 & -2 \\
 0 & 0 & 1 & 1 & 0 & 2 & 0 & 0 & -1 & -1 & -1 & -2 \\
 0 & 0 & 0 & 1 & 1 & 2 & -1 & 0 & 0 & -1 & -1 & -2 \\
 1 & 0 & 0 & 0 & 1 & 2 & -1 & -1 & 0 & 0 & -1 & -2 \\
 1 & 1 & 0 & 0 & 0 & 2 & -1 & -1 & -1 & 0 & 0 & -2 \\\hdashline[3pt/2pt]
 1 & 1 & 1 & 1 & 1 & 1 & 1 & 1 & 1 & 1 & 1 & 1 \\\hdashline[3pt/2pt]
 1 & 1 & 1 & 1 & 1 & 4 & -1 & -1 & -1 & -1 & -1 & -4 \\
\end{array}
\right).
\]

\ We will show that det$(Q)=-1$. For a given matrix, we denote by
$(Ci)$ its $i$-th column. When we say that, apply $(Ci)+(C1)$ on
the $i$-th column, we mean the column operation by adding the 1-st
column to the $i$-th column.

On the $(2n+2)$-th column, apply $C(2n+2)+C(n+1)$. On the
$(n+1)$-th column, apply $C(n+1)+\frac{n-1}{2}(C(2n+2))$. On the
$(n+1)$-th column, apply $C(n+1)-(C1)-\cdots-(Cn)$. Then all the
entries of the $(n+1)$-th column are zero except for the last
$(2n+2)$-th entry, which is $-1$. Hence it is enough to show
$\det(Q_0)=1$ where $Q_0$ is a $(2n+1)\times(2n+1)$ matrix defined
by
\begin{align*}
Q_0=\left(\begin{array}{@{}ccc;{3pt/2pt}ccc;{3pt/2pt}c@{}}
 & &  & & & & 1\\
\multicolumn{3}{c;{3pt/2pt}}{{\rm Circ}(\overbrace{1,\ldots,1}^\frac{n-1}{2},
\overbrace{0,\ldots,0}^\frac{n+1}{2})} &
\multicolumn{3}{c;{3pt/2pt}}{{\rm Circ}(\overbrace{1,\ldots,1}^\frac{n+1}{2},
\overbrace{0,\ldots,0}^\frac{n-1}{2})} & \vdots\\
 & & & & & & 1\\\hdashline[3pt/2pt]
 & & & & & & 0\\
\multicolumn{3}{c;{3pt/2pt}}{{\rm Circ}(\overbrace{0,\ldots,0}^\frac{n+1}{2},
\overbrace{1,\ldots,1}^\frac{n-1}{2})} &
\multicolumn{3}{c;{3pt/2pt}}{{\rm Circ}(\overbrace{0,\ldots,0}^\frac{n-1}{2},
\overbrace{-1,\ldots,-1}^\frac{n+1}{2})} & \vdots\\
 & & & & & & 0 \\\hdashline[3pt/2pt]
\multicolumn{3}{c;{3pt/2pt}}{\underbrace{1\qquad\quad\ \cdots\qquad\quad\ 1}_{n}}
 & \multicolumn{3}{c;{3pt/2pt}}{\underbrace{1\qquad\quad\ \cdots\qquad\quad\ 1}_{n}}  & 2
\end{array}\right)
\begin{array}{c}
 \Bigg\}\ n \\ \\ \\ \Bigg\}\ n \\ \\ \\
\end{array}.
\end{align*}

On the $(n+i)$-th column, apply $C(n+i)+C(f(\frac{n+1}{2}+i))$ for
$i=1,\ldots n$ where
\[
f(k)=
\begin{cases}
k & k \leq n \\
k-n & k>n.
\end{cases}
\]
On the $(n+i)$-th column, apply $C(n+i)-C(2n+1)$ for
$i=1,\ldots,n$. On the $(2n+1)$-th column, apply
$C(2n+1)-\frac{2}{n-1}\{(C1)+\cdots+(Cn)\}$. Thus we get
$\det(Q_0)=\det(Q_1)$ where
\begin{align*}
Q_1=\left(\begin{array}{@{}ccc;{3pt/2pt}ccc;{3pt/2pt}c@{}}
 & &  & & & & 0\\
\multicolumn{3}{c;{3pt/2pt}}{{\rm Circ}(\overbrace{1,\ldots,1}^\frac{n-1}{2},
\overbrace{0,\ldots,0}^\frac{n+1}{2})} &
&  & \mathbf{O} &\vdots\\
 & & & & & & 0\\\hdashline[3pt/2pt]
 & & & & & & 0\\
\multicolumn{3}{c;{3pt/2pt}}{{\rm Circ}(\overbrace{0,\ldots,0}^\frac{n+1}{2},
\overbrace{1,\ldots,1}^\frac{n-1}{2})} &
\multicolumn{3}{c;{3pt/2pt}}{{\rm Circ}(0,\overbrace{1,\ldots,1}^\frac{n-3}{2},0,
\overbrace{-1,\ldots,-1}^\frac{n-1}{2})} & \vdots\\
 & & & & & & 0 \\\hdashline[3pt/2pt]
\multicolumn{3}{c;{3pt/2pt}}{\underbrace{1\qquad\quad \cdots\qquad\quad 1}_{n}}
 & \multicolumn{3}{c;{3pt/2pt}}{\underbrace{0\qquad\qquad \cdots\qquad\qquad 0}_{n}}  & -\frac{2}{n-1}
\end{array}\right)
\begin{array}{c}
 \Bigg\}\ n \\ \\ \\ \Bigg\}\ n \\ \\ \\
\end{array}.
\end{align*}
Because of Lemma \ref{lemcirc} and
\[
\det({\rm Circ}(0,\overbrace{1,\ldots,1}^\frac{n-3}{2},0,
\overbrace{-1,\ldots,-1}^\frac{n-1}{2})) =\det({\rm
Circ}(\overbrace{-1,\ldots,-1}^\frac{n-1}{2},0,\overbrace{1,\ldots,1}^\frac{n-3}{2},0)),
\] we find $\det(Q_1)=1$.
\end{proof}


\begin{prop} \label{p3.7}
Let $G \simeq D_n$ where $n$ is an odd integer. The all the flabby
$G$-lattices are invertible. Consequently, if $k$ is a field
admitting a $D_n$-extension, then all the $G$-tori over $k$ are
retract $k$-rational.
\end{prop}

\begin{proof}
Since all the Sylow subgroups of $G$ are cyclic, the flabby
$G$-lattices are invertible by Theorem \ref{t2.7}. For a $G$-torus
over $k$, its function field is $K(M)^G$ for some $G$-lattice $M$,
some $G$-extension $K/k$. Since $[M]^{fl}$ is invertible, we may
apply Theorem \ref{t2.6}.
\end{proof}

\section{\boldmath Integral representations of $D_p$}

Let $G=\langle\sigma,\tau:\sigma^p=\tau^2=1,\tau\sigma\tau^{-1}=\sigma^{-1}\rangle\simeq D_p$ where $p$ is an odd prime number.
Define $H=\langle\tau\rangle$.

Denote $\zeta_p$ a primitive $p$-th root of unity,
$R=\bm{Z}[\zeta_p]$, $R_0=\bm{Z}[\zeta_p+\zeta^{-1}_p]$, $h_p^+$
the class number of $R_0$, $P=\langle 1-\zeta_p\rangle$ the unique
maximal ideal of $R$ lying over $\langle p\rangle\subset \bm{Z}$.
We may regard $R$ as a $G$-lattice by defining, for any $\alpha\in
R$, $\sigma\cdot\alpha=\zeta_p\alpha$,
$\tau\cdot\alpha=\bar{\alpha}$ the complex conjugate of $\alpha$.
Note that $R$ is a $G$-lattice of rank $p-1$; it is even a lattice
over $\bm{Z}[G]/\Phi_p(\sigma)\simeq \bm{Z}[\zeta]\circ H$, the
twisted group ring defined in Definition \ref{d3.2}.

If $I\subset R$ is an ideal with $\sigma(I)\subset I$, $\tau(I)\subset I$,
then $I$ may be regarded as a $G$-lattice also.
In particular,
if $\c{A}\subset R_0$ is an ideal,
then $R\c{A}$, $P\c{A}$ are $G$-lattices of rank $p-1$.

A complete list of non-isomorphic indecomposable $G$-lattices was
proved by Myrna Pike Lee \cite{Le}. In the following we adopt the
reformulation of Lee's Theorem in \cite[page 752, Theorem
(34.51)]{CR}. In the following theorem $\bm{Z}_-$ is the
$G$-lattice on which $\sigma$ acts trivially, and $\tau$ acts as
multiplication by $-1$.

\begin{theorem}[{M.\ P.\ Lee \cite[page 752]{Le,CR}}] \label{t4.1}
Let $G\simeq D_p$ where $p$ is an odd prime number. Let $\c{A}$
range over a full set of representatives of the $h_p^+$ ideal
classes of $R_0$ where $R_0=\bm{Z}[\zeta_p+\zeta_p^{-1}]$,
$R=\bm{Z}[\zeta_p]$, $P=\langle 1-\zeta_p\rangle$. Then there are
precisely $7h_p^++3$ isomorphism classes of indecomposable
$G$-lattices, and there are represented by
\[
~ \bm{Z},~ \bm{Z}_-,~ \bm{Z}[H]\simeq
\bm{Z}[G/\langle\sigma\rangle]; ~ R\c{A},~ P\c{A};
\]
and the non-split extensions
\begin{align*}
0 &\to P\c{A} \to V_{\c{A}}\to \bm{Z}\to 0,& 0 &\to R\c{A}\to X_{\c{A}} \to \bm{Z}_-\to 0, \\
0 &\to R\c{A} \to (Y_0)_{\c{A}} \to \bm{Z}[H]\to 0,& 0 &\to P\c{A}\to (Y_1)_{\c{A}}\to \bm{Z}[H]\to 0, \\
0 &\to R\c{A}\oplus P \to (Y_2)_{\c{A}} \to \bm{Z}[H]\to 0.
\end{align*}
\end{theorem}

\begin{remark}
In the above theorem, the words ``the non-split extensions $0 \to
P\c{A} \to V_{\c{A}}\to \bm{Z}\to 0$, ....." means that, if $M$ is
an indecomposable $G$-lattice satisfying that $0 \to P\c{A} \to M
\to \bm{Z}\to 0$, then $M \simeq V_{\c{A}}$ as $G$-lattices, i.e.
there is essentially a unique indecomposable lattice arising from
an extension of $\bm{Z}$ by $P\c{A}$. See \cite[pages 711-730]{CR}
and the proof in Lee's paper \cite{Le}.
\end{remark}

\begin{defn} \label{d4.2}
In Theorem \ref{t4.1}, when $\c{A}$ is a principal ideal in $R_0$,
we will write the corresponding $G$-lattices by $R$, $P$, $0\to
P\to V\to \bm{Z} \to 0$, $0\to R\to X\to \bm{Z}\to 0$, $0\to R\to
Y_0 \to \bm{Z}[H] \to 0$, $0\to P\to Y_1 \to \bm{Z}[H]\to 0$,
$0\to R\oplus P\to Y_2 \to \bm{Z}[H]\to 0$.

If $l$ is a prime number of $\bm{Z}$, denote by $\bm{Z}_l=\{m/n:
m,n\in\bm{Z}, l\nmid n\}$ the localization of $\bm{Z}$ at the
prime ideal $\langle l\rangle$. Since $(R_0)_l=\bm{Z}_l
[\zeta_p+\zeta_p^{-1}]$ is a semi-local principal ideal domain, we
find that $\c{A}_l$ is a principal ideal in $(R_0)_l$ for any
prime number $l$.

It follows that, if $\c{A}$ is any ideal in $R_0$, then $R$ and
$R\c{A}$, $P$ and $P\c{A}$, $V$ and $(V)_{\c{A}}$, $X$ and
$X_{\c{A}}$, ... belong to the same genus, i.e. they become
isomorphic after localization at any prime ideal $\langle
l\rangle$ of $\bm{Z}$ (see \cite[page 642]{CR}).

\end{defn}

We will show that $M_+$, $M_-$, $N_+$, $N_-$, $\widetilde{M}_+$,
$\widetilde{M}_-$ defined in Section 3 are isomorphic to $V$, $X$,
$R$, $P$, $Y_0$, $Y_1$ when $n=p$ is an odd prime number.

\begin{lemma} \label{l4.3}
Let $N_+$ and $N_-$ be the $G$-lattices with $G=\langle
\sigma,\tau:\sigma^n=\tau^2=1,\tau\sigma\tau^{-1}=\sigma^{-1}\rangle
\simeq D_n$ where $n$ is an odd integer. If $n=p$ is an odd prime
number, then $N_+\simeq R$ and $N_-\simeq P$.
\end{lemma}

\begin{proof}
By definition, when $n=p$, $N_+$ has a $\bm{Z}$-basis $\{u_i:1\le
i\le p-1\}$ with $\sigma: u_1\mapsto u_2\mapsto \cdots \mapsto
u_{p-1}\mapsto -(u_1+u_2+\cdots+u_{p-1})$,
$\tau:u_i\leftrightarrow u_{p-i}$. On the other hand,
$R=\sum_{0\le i\le p-1} \bm{Z}\cdot \zeta_p^i$ has a
$\bm{Z}$-basis $\{\zeta_p^i:1\le i\le p-1\}$ with
$\sigma:\zeta_p\mapsto \zeta_p^2 \mapsto\cdots\mapsto
\zeta_p^{p-1}\mapsto -(\zeta_p+\cdots+\zeta_p^{p-1})$,
$\tau:\zeta_p^i\leftrightarrow \zeta_p^{p-i}$. Hence the result.

For the proof of $N_-\simeq P$, note that
$P=R(1-\zeta_p)=\sum_{0\le i\le p-1}
\bm{Z}(\zeta_p^i-\zeta_p^{i+1})$. Define
$v_0=\zeta_p^{\frac{p-1}{2}} -\zeta_p^{\frac{p+1}{2}}$,
$v_i=\sigma^i(v_0)$ for $0\le i\le p-1$. Then
$\{v_1,v_2,\ldots,v_{p-1}\}$ is a $\bm{Z}$-basis of $P$ with
$\sigma: v_1\mapsto v_2\mapsto\cdots\mapsto v_{p-1}\mapsto
-(v_1+v_2+\cdots+v_{p-1})$ and $\tau: v_i\leftrightarrow -v_{p-i}$
because $\tau(\zeta_p)=\zeta_p^{-1}$. Thus $P\simeq N_-$.
\end{proof}


\begin{lemma} \label{l4.4}
Let $M_+$, $M_-$, $N_+$, $N_-$ be $G$-lattices with
$G=\langle\sigma,\tau:\sigma^n=\tau^2=1,
\tau\sigma\tau^{-1}=\sigma^{-1}\rangle\simeq D_n$ where $n$ is an
odd integer. Then there are non-split exact sequences of
$G$-lattices $0\to N_-\to M_+\to \bm{Z}\to 0$, $0\to N_+\to M_-\to
\bm{Z}_-\to 0$. When $n=p$ is an odd prime number, then $M_+\simeq
V$, $M_-\simeq X$.
\end{lemma}

\begin{proof}
\begin{Case}{1} $M_+$. \end{Case}

By definition, choose a $\bm{Z}$-basis $\{x_i:0\le i\le n-1\}$ of
$M_+$ such that $\sigma: x_i\mapsto x_{i+1}$,
$\tau:x_i\leftrightarrow x_{n-i}$ where the index is understood
modulo $n$.

Define $u_0=x_{\frac{n-1}{2}}-x_{\frac{n+1}{2}}$,
$t=x_{\frac{n-1}{2}}$, $u_i=\sigma^i(u_0)$ for $0\le i\le n-1$.

It follows that $\sum_{0\le i\le n-1} u_i=0$ and
$\{u_1,u_2,\ldots,u_{n-1},t\}$ is a $\bm{Z}$-basis of $M_+$ with
$\sigma$ and $\tau$ acting by
\begin{align*}
\sigma :{}& u_1\mapsto u_2\mapsto \cdots \mapsto u_{n-1}\mapsto u_0=-(u_1+u_2+\cdots+u_{n-1}), \\
& t\mapsto t+u_1+u_2+\cdots+u_{n-1}, \\
\tau :{}& u_i\leftrightarrow-u_{n-i},~ t\mapsto
t+u_1+u_2+\cdots+u_{n-1}.
\end{align*}

Note that $\sum_{1\le i\le n-1} \bm{Z}\cdot u_i\simeq N_-$ and
$M_+/(\sum_{1\le i\le n-1} \bm{Z}\cdot u_i)\simeq \bm{Z}$. Hence
we get the sequence $0\to N_-\to M_+\to \bm{Z} \to 0$.

This sequence doesn't split. Otherwise, there is some element
$s\in M_+$ such that $\sigma(s)=\tau(s)=s$, and
$\{u_1,u_2,\ldots,u_{n-1},s\}$ is a $\bm{Z}$-basis of $M_+$.

Write $s=\sum_{1\le i\le n-1} a_i\cdot u_i+b\cdot t$ where $a_i,b
\in \bm{Z}$. Because $\{u_1,\ldots,u_{n-1},s\}$ is a
$\bm{Z}$-basis of $M_+$, we find that $b=\pm 1$.

Consider the case $b=-1$ (the situation $b=1$ can be discussed
similarly). Since $\tau(\sum_{1\le i\le n-1} a_i u_i-t)=\sum_{1\le
i\le n-1} a_iu_i-t$, we find that $a_i-1=a_{n-i}$ and
$a_{n-i}-1=a_i$ for all $1\le i\le n-1$. This is impossible.

Now assume that $n=p$ is an odd prime number. We will show that
$M_+\simeq V$.

By Lemma \ref{l4.3}, $N_-\simeq P$. Thus we have a non-split
extension $0\to P\to M_+\to \bm{Z}\to 0$. Then apply the remark
after Theorem \ref{t4.1}. More precisely, it is proved in
\cite[page 221]{Le} that, up to $G$-lattice isomorphisms, there is
precisely one indecomposable $G$-lattice arising from extensions
of $\bm{Z}$ by $P$, although $\fn{Ext}_{\bm{Z}[G]}^1
(\bm{Z},P)=\bm{Z}/p\bm{Z}$ by \cite[Lemma 2.1]{Le}. Since $0\to
P\to V\to \bm{Z}\to 0$ is a non-split extension by Theorem
\ref{t4.1}, we conclude that $M_+\simeq V$.

\begin{Case}{2} $M_-$. \end{Case}

The proof is similar to Case 1. Choose a $\bm{Z}$-basis $\{x_i:
0\le i\le n-1\}$ of $M_-$ with $\sigma: x_i\mapsto x_{i+1}$,
$\tau: x_i\mapsto -x_{n-i}$.

Define $u_0=x_{\frac{n-1}{2}}-x_{\frac{n+1}{2}}$,
$t=x_{\frac{n-1}{2}}$, $u_i=\sigma^i (u_0)$ for $0\le i\le n-1$.
We find that
\begin{align*}
\sigma :{}& u_1\mapsto u_2\mapsto \cdots \mapsto u_{n-1}\mapsto -(u_1+u_2+\cdots+u_{n-1}), \\
& t\mapsto t+u_1+u_2+\cdots+u_{n-1}, \\
\tau :{}& u_i\leftrightarrow u_{n-i},~ t\mapsto
-t-u_1-u_2-\cdots-u_{n-1}.
\end{align*}

Thus $\sum_{0\le i\le n-1} \bm{Z}\cdot u_i \simeq N_+$ and
$M_-/(\sum_{1\le i\le n-1} \bm{Z}u_i)\simeq \bm{Z}_-$.

Similarly, the sequence $0\to N_+\to M_-\to \bm{Z}_-\to 0$ doesn't
split.

The proof of $M_-\simeq X$ when $n=p$ is a prime number is the
same.
\end{proof}

\begin{lemma} \label{l4.5}
Let $N_+$, $N_-$, $\widetilde{M}_+$, $\widetilde{M}_-$ be
$G$-lattices with
$G=\langle\sigma,\tau:\sigma^n=\tau^2=1,\tau\sigma\tau^{-1}=\sigma^{-1}\rangle
\simeq D_n$ where $n$ is an odd integer. Then there are non-split
exact sequences of $G$-lattices $0\to N_+\to \widetilde{M}_- \to
\bm{Z}[G/\langle\sigma\rangle]\to 0$, $0\to N_-\to
\widetilde{M}_+\to \bm{Z}[G/\langle\sigma\rangle]\to 0$. When
$n=p$ is an odd prime number, then $\widetilde{M}_-\simeq Y_0$,
$\widetilde{M}_+\simeq Y_1$.
\end{lemma}

\begin{proof}
\begin{Case}{1} $\widetilde{M}_+$. \end{Case}

We adopt the same notations $x_0,x_1,\ldots,x_{n-1},u_1,\ldots,u_{n-1}$ in the proof of Lemma \ref{l4.4}.
Write $\widetilde{M}_+=(\bigoplus_{0\le i\le n-1} \bm{Z}\cdot x_i)\oplus \bm{Z}\cdot w$ with
\begin{align*}
\sigma &: x_i\mapsto x_{i+1}, ~ w\mapsto w, \\
\tau &: x_i \mapsto x_{n-i}, ~ w\mapsto -w+x_0+x_1+\cdots+x_{n-1}
\end{align*}
where the index is understood modulo $n$.

Define $u_0=x_{\frac{n-1}{2}} - x_{\frac{n+1}{2}}$, $t=x_{\frac{n-1}{2}}$, $u_i=\sigma^i (u_0)$ for $0\le i\le n-1$.

We claim that $\sum_{0\le i\le n-1} x_i=nt-(n-1)u_0-(n-2)u_1-\cdots-u_{n-2}$.

Since $x_{\frac{n-1}{2}}=t$, $x_{\frac{n+1}{2}}=t-u_0$, we find
that $x_{\frac{n+3}{2}}=x_{\frac{n+1}{2}}-u_1=t-u_0-u_1$. By
induction, we may find similar formulae for $x_{\frac{n+5}{2}},
\ldots,x_{n-1},x_0,\ldots,x_{\frac{n-3}{2}}$. In particular,
$x_{\frac{n-3}{2}}=t-u_0-u_1-\cdots -u_{n-2}$. Thus the formula of
$\sum_{0\le i\le n-1} x_i$ is found.

Note that $\{u_1,\ldots,u_{n-1},t,w\}$ is a $\bm{Z}$-basis of $\widetilde{M}_+$ and
\begin{align*}
\sigma &: u_1\mapsto u_2\mapsto \cdots \mapsto u_{n-1}\mapsto -(u_1+u_2+\cdots+u_{n-1}), ~ t\mapsto t+\sum_{1\le i\le n-1} u_i,~ w\mapsto w, \\
\tau &: u_i\mapsto -u_{n-i},~ t\mapsto t+\sum_{1\le i\le n-1} u_i,~ w\mapsto -w+u_1+2u_2+\cdots+(n-1)u_{n-1}+nt.
\end{align*}

Define $w_0=-\frac{n-1}{2}t+w$, $w_1=\frac{n+1}{2}t-w$.
Then $\{u_1,\ldots,u_{n-1},w_0,w_1\}$ is also a $\bm{Z}$-basis of $\widetilde{M}_+$ with
\begin{align*}
\sigma:{}& w_0\mapsto w_0-\frac{n-1}{2}\!\sum_{1\le i\le n-1}\! u_i,~w_1\mapsto w_1+\frac{n+1}{2}\!\sum_{1\le i\le n-1}\! u_i, \\
\tau:{}& w_0\mapsto w_1-\tfrac{n-3}{2}u_1-\tfrac{n-5}{2}u_2-\cdots-u_{\frac{n-3}{2}}+u_{\frac{n+1}{2}}+2u_{\frac{n+3}{2}}+\cdots+\tfrac{n-1}{2}u_{n-1}, \\
& w_1\mapsto w_0+\tfrac{n-1}{2}u_1+\tfrac{n-3}{2}u_2+\cdots+u_{\frac{n-1}{2}}-u_{\frac{n+3}{2}}-2u_{\frac{n+5}{2}}-\cdots-\tfrac{n-3}{2}u_{n-1}.
\end{align*}

Note that $\sum_{1\le i\le n-1}\bm{Z}\cdot u_i \simeq N_-$ and $\widetilde{M}_+/(\sum_{1\le i\le n-1} \bm{Z}\cdot u_i)\simeq \bm{Z}[G/\langle\sigma\rangle]$.
It follows that we get an exact sequence $0\to N_-\to \widetilde{M}_+\to \bm{Z}[G/\langle\sigma\rangle]\to 0$.

\medskip
We will show that this exact sequence doesn't split.

Suppose not.
Then there exists $s\in \widetilde{M}_+$ such that $\sigma(s)=s$ and $\{u_1,\ldots,u_{n-1},s,\tau(s)\}$ is a $\bm{Z}$-basis of $\widetilde{M}_+$.

Write $s=\sum_{1\le i\le n-1} a_iu_i+b_0w_0+b_1w_1$ where $a_i,b_j\in \bm{Z}$.
Since $\tau(s)=\sum_{1\le i\le n-1}a'_i u_i+b_1w_0+b_0w_1$ for some integers $a'_i\in \bm{Z}$,
it follows that $b_0^2-b_1^2=\pm 1$ (remember that $\{u_1,\ldots,u_{n-1},s,\tau(s)\}$ is a $\bm{Z}$-basis of $\widetilde{M}_+$).
It follows that the only solutions for the pair $(b_0,b_1)$ are $(b_0,b_1)=(\pm 1,0)$, $(0,\pm 1)$.

We consider the situation $(b_0,b_1)=(1,0)$ (the other situations
may be discussed similarly). Write $s=\sum_{1\le i\le
n-1}a_iu_i+w_0$ as before. Since $\sigma(s)=s$, we find an
identity of the ordered $(n-1)$-tuples :
$(a_1,a_2,\ldots,a_{n-1})=(0,a_1,a_2,\ldots,a_{n-2})-a_{n-1}(1,1,\ldots,1)-\frac{n-1}{2}(1,1,\ldots,1)$.
Solve $a_1,a_2,\ldots,a_{n-1}$ inductively in terms of $a_{n-1}$.
We find $a_1=-(a_{n-1}+\frac{n-1}{2})$,
$a_2=-2(a_{n-1}+\frac{n-1}{2})$, $\ldots$,
$a_{n-1}=-(n-1)(a_{n-1}+\frac{n-1}{2})$. Hence
$na_{n-1}=-(n-1)^2/2$. But $\fn{gcd}\{n,n-1\}=1$. Thus we find a
contradiction.

In conclusion, $0\to N_-\to \widetilde{M}_+\to \bm{Z}[G/\langle\sigma\rangle]\to 0$ doesn't split.

When $n=p$ is an odd prime number, we get a non-split extension
$0\to P\to \widetilde{M}_+\to \bm{Z}[G/\langle\sigma\rangle]\to
0$. It is proved by Lee (see the last paragraph of \cite[page
221]{Le}) that the non-split extensions of
$\bm{Z}[G/\langle\sigma\rangle]$ by $P$ give rise to precisely one
indecomposable $G$-lattice, although $\fn{Ext}^1_{\bm{Z}[G]}
(\bm{Z}[G/\langle\sigma\rangle],P)\simeq \bm{Z}/p\bm{Z}$ by
\cite[Lemma 2.1]{Le}. Since $0\to P\to Y_1\to
\bm{Z}[G/\langle\sigma\rangle]\to 0$ is also a non-split
extension, we conclude that $\widetilde{M}_+\simeq Y_1$.

\medskip
\begin{Case}{2} $\widetilde{M}_-$. \end{Case}

The proof is similar.
We adopt the notations $x_0,x_1,\ldots,x_{n-1},u_1,\ldots,u_{n-1}$ in the proof of Case 1.
Write $\widetilde{M}_- =(\bigoplus_{0\le i\le n-1} \bm{Z}\cdot x_i)\oplus \bm{Z}\cdot w$ such that
\begin{align*}
\sigma &: x_i\mapsto x_{i+1},~ w\mapsto w, \\
\tau &: x_i\mapsto -x_{n-i},~ w\mapsto w-\sum_{0\le i\le n-1} x_i
\end{align*}
where $0\le i\le n-1$ and the index is understood modulo $n$.

Define $u_0=x_{\frac{n-1}{2}}-x_{\frac{n+1}{2}}$, $t=x_{\frac{n-1}{2}}$, $u_i=\sigma^i(u_0)$ for $0\le i\le n-1$.
Then $\{u_1,\ldots,u_{n-1},t,w\}$ is a $\bm{Z}$-basis of $\widetilde{M}_-$ and
\begin{align*}
\sigma:{}& u_1\mapsto u_2\mapsto\cdots\mapsto u_{n-1}\mapsto -(u_1{+}\cdots{+}u_{n-1}),\\
& t\mapsto t+u_1+u_2+\cdots+u_{n-1},\, w\mapsto w, \\
\tau:{}& u_i\mapsto u_{n-i},~ t\mapsto -t-\left(\sum_{1\le i\le n-1} u_i\right),\\
& w\mapsto w-u_1-2u_2-\cdots-(n-1)u_{n-1}-nt.
\end{align*}

Define $w_0=\frac{n-1}{2}t-w$, $w_1=\frac{n+1}{2}-w$.
We find that
\begin{align*}
\sigma:{}& w_0\mapsto w_0+\frac{n-1}{2}\sum_{1\le i\le n-1} u_i,~ w_1\mapsto w_1+\frac{n+1}{2}\sum_{1\le i\le n-1}u_i, \\
\tau:{}& w_0\mapsto w_1-\tfrac{n-3}{2}u_1-\tfrac{n-5}{2}u_2-\cdots-u_{\frac{n-3}{2}}+u_{\frac{n+1}{2}}+2u_{\frac{n+3}{2}}+\cdots+\tfrac{n-1}{2}u_{n-1}, \\
& w_1\mapsto w_0-\tfrac{n-1}{2}u_1-\tfrac{n-3}{2}u_2-\cdots-u_{\frac{n-1}{2}}+u_{\frac{n+3}{2}}+2u_{\frac{n+5}{2}}+\cdots+\tfrac{n-3}{2}u_{n-1}.
\end{align*}

The remaining proof is similar and is omitted.
\end{proof}

\begin{lemma} \label{l4.6}
Let $N_+$, $N_-$, be $G$-lattices with
$G=\langle\sigma,\tau:\sigma^n=\tau^2=1,\tau\sigma\tau^{-1}=\sigma^{-1}\rangle
\simeq D_n$ where $n$ is an odd integer. Then there are non-split
exact sequences of $G$-lattices $0\to N_+ \oplus N_- \to
\bm{Z}[G]\to \bm{Z}[G/\langle\tau\rangle] \to 0$. When $n=p$ is an
odd prime number, then $Y_2\simeq \bm{Z}[G]$.
\end{lemma}

\begin{proof}
This lemma was proved by Lee for the case when $n=p$ is an odd
prime number in (i) of Case 1 of \cite[pages 222--224]{Le}. There
was also a remark in the first paragraph of \cite[page 229,
Section 4]{Le}.

Here is a proof when $n$ is an odd integer. Once the first part is
proved, we may deduce the second part when $n=p$ is an odd prime
number because $N_+ \simeq R$, $N_- \simeq P$ (by Lemma
\ref{l4.3}) and there is a unique indecomposable $G$-lattice
arising from non-split extensions of
$\bm{Z}[G/\langle\tau\rangle]$ by $R \oplus P$ (see \cite[page
222]{Le}). Hence $\bm{Z}[G] \simeq Y_2$.

Now we start to prove the first part with
$G=\langle\sigma,\tau:\sigma^n=\tau^2=1,\tau\sigma\tau^{-1}=\sigma^{-1}\rangle
\simeq D_n$ where $n$ is an odd integer.

{}From now on till the end of the proof, denote $\zeta=\zeta_n$ a
primitive $n$-th root of unity, $R=\bm{Z}[\zeta]$,
$R_0=\bm{Z}[\zeta+\zeta^{-1}]$, $H=\langle\tau\rangle$.

\bigskip
Step 1. Let $\{\sigma^i,\sigma^i \tau: 0\le i\le n-1\}$ be a
$\bm{Z}$-basis of $\bm{Z}[G]$.

Let $\{t_0,t_1\}$ be a $\bm{Z}$-basis of $\bm{Z}[H]$ with
$\sigma(t_i)=t_i$, $\tau:t_0\leftrightarrow t_1$.

Define a $G$-lattice surjection $\varphi:\bm{Z}[G]\to \bm{Z}[H]$
by $\varphi(\sigma^i)=t_0$, $\varphi(\sigma^i \tau)=t_1$. Define a
$G$-lattice $M$ by $M=\fn{Ker}(\varphi)$. We will prove that $M
\simeq N_+ \oplus N_-$ (note that $\bm{Z}[G]$ is indecomposable
\cite{Sw4}).

Define $u_i, v_i \in M$ as follows. Define $u_0=
\sigma^{(n-1)/2}-\sigma^{(n+1)/2}$, $v_0=\sigma^{(n+1)/2}
\tau-\sigma^{(n-1)/2} \tau$, and $u_i= \sigma^i(u_0)$, $v_i=
\sigma^i(v_0)$ for $0 \le i \le n-1$.

It follows that $\sum_{0 \le i \le n-1} u_i = \sum_{0 \le i \le
n-1} v_i=0$, and $\{u_i,v_i : 1\le i\le n-1\}$ is a $\bm{Z}$-basis
of $M$. Moreover, it is easy to see that $\sigma: u_i \mapsto
u_{i+1},v_i \mapsto v_{i+1}$, $\tau: u_i \mapsto v_{n-i},v_i
\mapsto u_{n-i}$ where the index is understood modulo $n$.

\bigskip
Step 2

Define $x_i=u_i+v_i$, $y_i=u_{i-1}-v_{i+1}$ where $0\le i \le
n-1$. Clearly $\sum_{0 \le i \le n-1}x_i=\sum_{0 \le i \le
n-1}y_i=0$. We claim that $\{x_i, y_i: 1 \le i \le n-1 \}$ is a
$\bm{Z}$-basis of $M$.

Assume the above claim. Define $M_1=\oplus_{1 \le i \le n-1}
\bm{Z} \cdot x_i$, $M_2=\oplus_{1 \le i \le n-1} \bm{Z} \cdot
y_i$. It is easy to verify that $M_1 \simeq N_+$ and $M_2 \simeq
N_-$. Hence the proof that $M \simeq N_+ \oplus N_-$ is finished.

\bigskip
Step 3

We will prove that $\{x_i, y_i: 1 \le i \le n-1 \}$ is a
$\bm{Z}$-basis of $M$.

Let $Q$ be the coefficient matrix of $x_1, x_2,
\ldots,x_{n-1},y_1, \ldots,y_{n-1}$ with respect to the
$\bm{Z}$-basis $u_1, u_2, \ldots,u_{n-1},v_1, \ldots,v_{n-1}$. For
the sake of visual convenience, we will consider the matrix $P$
which is the transpose of $Q$. We will show that det$(P)= 1$.


The matrix $P$ is defined as
\[
P=\left(\begin{array}{@{}cccc;{3pt/2pt}ccccc@{}}
1 &  &  &  & 1 &  &  & \\
 & \ddots &  &  &  &  \ddots & & \\
 &  & \ddots &  &  &  & \ddots & \\
 &  &  & 1 &  &  & & 1 \\\hdashline[3pt/2pt]
-1 & \cdots & -1 & -1 & 0 & -1 & & \\
1 &  &  & 0 & \vdots &  & \ddots & \\
 &  \ddots &  & \vdots & 0 &  &  & -1 \\
 &   & 1 &  0 & 1 & 1 &  \cdots & 1
\end{array}\right).
\]

For examples, when $n=3, 5$, it is of the form

\begin{align*}
P&=\left(\begin{array}{@{}cc;{3pt/2pt}ccc@{}}
1 & 0 & 1 & 0\\
0 & 1 & 0 & 1\\\hdashline[3pt/2pt]
-1 & -1 & 0 & -1\\
1 &  0 & 1 & 1
\end{array}\right),\\
P&=\left(\begin{array}{@{}cccc;{3pt/2pt}ccccc@{}}
1 & 0 & 0 & 0 & 1 & 0 & 0 & 0\\
0 & 1 & 0 & 0 & 0 & 1 & 0 & 0\\
0 & 0 & 1 & 0 & 0 & 0 & 1 & 0\\
0 & 0 & 0 & 1 & 0 & 0 & 0 & 1 \\\hdashline[3pt/2pt]
-1 & -1 & -1 & -1 & 0 & -1 & 0 & 0\\
1 & 0 & 0 & 0 & 0 & 0 & -1 & 0\\
0 & 1 & 0 & 0 & 0 & 0 & 0 & -1 \\
0 & 0 & 1 &  0 & 1 & 1 &  1 & 1
\end{array}\right).
\end{align*}

In the case $n=3,5$, it is routine to show that det$(P)= 1$. When
$n \ge 7$, we will apply column operations on the matrix $P$ and
then expand the determinant along a row. Thus we are reduced to
matrices of smaller size.

For a given matrix, we denote by $(Ci)$ 
its $i$-th column.
When we say that, apply $(Ci)+(C1)$
on the $i$-th column, we mean the column operation by adding the
1-st column to the $i$-th column.

\bigskip
Step 4

We will prove det$(P)= 1$ where $P$ is the $(2n-2)\times(2n-2)$
integral matrix defined in Step 3. Suppose $n \ge 7$.

Apply column operations on the matrix $P$. On the $(n+i)$-th
column where $0 \le i \le n-2$, apply $C(n+i)-C(i+1)$.

Thus all the entries of the right upper part of the resulting
matrix vanish. We get det$(P)=$ det$(P_0)$ where $P_0$ is an
$(n-1)\times(n-1)$ integral matrix defined as

\begin{align*}
P_0&=\left(\begin{array}{@{}cccccccc@{}}
1 & 0 & 1 & 1 & 1 & 1 & \cdots & 1\\
-1 & 0 & -1 & & & & &  \\
& -1 & 0 & -1 & & & &  \\
& & -1 & 0 & -1 & & &  \\
& & & -1 & 0 & -1 & &  \\
& & &  & \ddots & \ddots & \ddots & \\
& & & &  & -1 & 0 & -1 \\
1 & 1 & 1 & \cdots& 1 & 1 & 0 & 1
\end{array}\right).
\end{align*}

\bigskip
Step 5

Apply column operations on $P_0$. On the 3rd column, apply
$(C3)-(C1)$. Then, on the 4th column, apply $(C4)-(C2)$.

In the resulting matrix, each of the 2nd row and the 3rd row have
only one non-zero entry.

Thus det$(P_0) = $ det$(P_1)$ where $P_1$ is an $(n-3)\times(n-3)$
integral matrix defined as

\begin{align*}
P_1&=\left(\begin{array}{@{}cccccccc@{}}
0 & 1 & 1 & 1 & 1 & \cdots & 1\\
-1 & 0 & -1 & & & & &  \\
& -1 & 0 & -1 & & & &  \\
& & -1 & 0 & -1 & & &  \\
 & &  & \ddots & \ddots & \ddots & \\
& & & &  -1 & 0 & -1 \\
0 & 0 & 1 & \cdots& 1 & 0 & 1
\end{array}\right).
\end{align*}

\bigskip
Step 6

Apply column operations on $P_1$. On the 3rd column, apply
$(C3)-(C1)$. On the 4th column, apply $(C4)-(C2)$.

Then expand the determinant along the 2nd row and the 3rd row. We
get det$(P_1) = $ det$(P_2)$ where $P_2$ is an $(n-5)\times(n-5)$
integral matrix defined as

\begin{align*}
P_2&=\left(\begin{array}{@{}ccccc@{}}
1 & 0 & 1 & \cdots & 1 \\
-1 & 0 & -1 &  &  \\
 & \ddots & \ddots & \ddots & \\
 &  & -1 & 0 & -1 \\
1 & \cdots & 1 & 0 & 1
\end{array}\right).
\end{align*}

But the matrix $P_2$ looks the same as $P_0$ except the size.
Done.

\end{proof}

\begin{lemma} \label{l4.7}
Let $M_1$ and $M_2$ be $G$-lattices belonging to the same genus.
Then $M_1$ is flabby if and only if so is $M_2$.
\end{lemma}

\begin{proof}
For any subgroup $S$ of $G$, $H^{-1}(S,M_i)$ is a finite abelian group.
If $H^{-1}(S,M_1)\allowbreak\ne 0$ for some subgroup $S$,
then there is some prime number $l$ such
that $H^{-1}(S,\allowbreak\bm{Z}_l[G]\otimes_{\bm{Z}[G]}M_1)\ne 0$ because localization commutes with taking Tate cohomology.
Thus $H^{-1}(S,M_2)\ne 0$
because $\bm{Z}_l[G]\otimes_{\bm{Z}[G]} M_1\simeq \bm{Z}_l[G]\otimes_{\bm{Z}[G]}M_2$.

We thank Kunyavskii for providing the following alternative proof.
By Roiter's Theorem \cite[page 660]{CR}, there is a rank-one
projective module $L$ over $\bm{Z}[G]$ such that $M_1 \oplus L
\simeq M_2 \oplus \bm{Z}[G]$. Hence the result.
\end{proof}

\begin{prop} \label{p4.8}
Let the notations be the same as in Theorem \ref{t4.1}.
Then the indecomposable $G$-lattices which are flabby are
\[
\bm{Z},~\bm{Z}[H],~V_{\c{A}},~ (Y_0)_{\c{A}},~ (Y_1)_{\c{A}},~ (Y_2)_{\c{A}},
\]
while the remaining ones are not flabby.
\end{prop}

\begin{proof}
Remember that $V_{\c{A}}$ and $V$ belonging to the same genus (see Definition \ref{d4.2}).
Apply Lemma \ref{l4.7}.
It suffices to check whether the following ten lattices
\[
\bm{Z},~\bm{Z}[H],~ V,~ Y_0,~ Y_1,~ Y_2,~ \bm{Z}_-,~ R,~ P,~ X
\]
are flabby lattices.

Since $\bm{Z}$, $\bm{Z}[H]$ are permutation lattices, they are flabby.

By Lemma \ref{l4.4}, Lemma \ref{l4.5} and Lemma \ref{l4.6},
we find that $V\simeq M_+$, $Y_0\simeq \widetilde{M}_-$, $Y_1\simeq \widetilde{M}_+$, $Y_2\simeq \bm{Z}[G]$.
$M_+$ and $\bm{Z}[G]$ are permutation lattices.
Hence they are flabby.
As to $\widetilde{M}_-$ and $\widetilde{M}_+$.
Applying Theorem \ref{t3.4} and Theorem \ref{t3.5},
we find that they are stably permutation.
Thus they are flabby also.

Now we turn to the non-flabby cases.
Since $G=D_p$, a flabby lattices is necessarily an invertible lattice by Theorem \ref{t2.7}; thus it is also coflabby.
In summary, if we want to show that a $G$-lattice $M$ is not flabby,
we may show that it is not coflabby or it is not invertible.

\medskip
For $\bm{Z}_-$, $H^1(G,\bm{Z}_-)=\bm{Z}/2\bm{Z}$. For, write
$\bm{Z}_-=\bm{Z}\cdot w$ with $\sigma(w)=w$, $\tau(w)=-w$. Apply
the Hochschild-Serre spectral sequence $0\to H^1(\langle
\tau\rangle, \bm{Z}_-^{\langle\sigma\rangle})\to H^1(G,\bm{Z}_-)
\to H^1(\langle\sigma\rangle,\bm{Z}_-)^{\langle\tau\rangle}$.
Because $H^1(\langle\sigma\rangle,\bm{Z}_-)=0$ and
$H^1(\langle\tau\rangle,\bm{Z}^{\langle\sigma\rangle}_-)=\bm{Z}/2\bm{Z}$,
we find that $H^1(G,\bm{Z}_-)=\bm{Z}/2\bm{Z}\ne 0$.

\medskip
For $R$, if $R$ is a flabby $G$-lattice, then it is invertible.
Restricted to the subgroup $S=\langle\sigma\rangle$, $R$ become an
invertible $S$-lattice.

{}From the short exact sequence of $S$-lattices $0\to R\to \bm{Z}[S]
\xrightarrow{\varepsilon} \bm{Z}\to 0$ where $\varepsilon$ is the
augmentation map, since $R$ is $S$-invertible and $\bm{Z}$ is
$S$-permutation, it follows that the sequence splits
\cite[Proposition 1.2]{Len}. Thus $\bm{Z}[S]\simeq R\oplus \bm{Z}$
is not indecomposable as an $S$-lattice. This leads to a
contradiction.

\medskip
For $P$, if $P$ is a flabby $G$-lattice, then it is
invertible. From the short exact sequence of $G$-lattices $0\to P\to
\bm{Z}[G/\langle\tau\rangle]\xrightarrow{\varepsilon} \bm{Z}\to 0$
where $\varepsilon$ is the augmentation map, since $P$ is
$G$-invertible and $\bm{Z}$ is $G$-permutation, the sequence
splits by \cite[Proposition 1.2]{Len}. Thus
$\bm{Z}[G/\langle\tau\rangle]$ is a decomposable $G$-lattice. But
$\bm{Z}[G/\langle\tau\rangle]=\fn{Ind}^G_{\langle\tau\rangle}\bm{Z}=M_+$
by Definition \ref{d3.1}. We find a contradiction again.

\medskip
For $X$, we know that $X\simeq M_-$ by Lemma \ref{l4.3}. Let
$S=\langle\tau\rangle$. Regard $M$ as an $S$-lattice. By
Definition \ref{d3.1}, as an $S$-lattice, $M_-\simeq N\oplus
\bm{Z}$ where $N$ is isomorphic to $\frac{p-1}{2}$ copies of
$\bm{Z}[S]$. Thus $H^1(S,M_-)=H^1(S,N\oplus \bm{Z}_-)\simeq
H^1(S,\bm{Z}_-)=\bm{Z}/2\bm{Z}\ne 0$. Hence $X\simeq M_-$ is not
coflabby.
\end{proof}

\begin{theorem} \label{t4.9}
Let $G=\langle\sigma,\tau:\sigma^p=\tau^2=1,\tau\sigma\tau^{-1}=\sigma^{-1}\rangle\simeq D_p$ where $p$ is an odd prime number.
Assume that $h_p^+=1$.
\begin{enumerate}
\item[{\rm (1)}]
Then there are precisely ten non-isomorphic indecomposable $G$-lattices
\[
\bm{Z},~\bm{Z}_-,~\bm{Z}[H],~R,~P,~V,~X,~Y_0,~Y_1,~Y_2.
\]

Among these lattices, $\bm{Z}$, $\bm{Z}[H]$, $V$, $Y_2$ are
permutation $G$-lattices, $Y_0$ and $Y_1$ are stably permutation
$G$-lattices, $\bm{Z}_-$, $R$, $P$, $X$ are not flabby
$G$-lattices. \item[{\rm (2)}] If $k$ is a field admitting a
$G$-extension, then any $G$-torus defined over $k$ is stably
$k$-rational. In other words, if $K/k$ is a Galois extension field
with $\fn{Gal}(K/k)\simeq G$ and $M$ is any $G$-lattices, then
$K(M)^G$ is stably $k$-rational.

\end{enumerate}
\end{theorem}

\begin{proof}
(1) follows from Theorem \ref{t4.1} and the proof of Proposition \ref{p4.8}.

For the proof of (2), apply Theorem \ref{t2.6}. It suffices to
show that $[M]^{fl}$ is permutation for any $G$-lattice $M$.

Choose any exact sequence of $G$-lattices $0\to M\to Q\to E\to 0$
where $Q$ is a permutation $G$-lattice and $E$ is a flabby
$G$-lattice. Write $E$ as a direct sum of indecomposable
$G$-lattices $E=\bigoplus_{1\le i\le m} E_i$ where each $E_i$ is
indecomposable. It is necessary that each $E_i$ is flabby. Apply
the result in (1), we find that $E$ is stably permutation, i.e.\
there is an permutation $G$-lattice $Q'$ such that $E\oplus Q'$ is
a permutation $G$-lattice. Thus we get an exact sequence $0\to
M\to Q\oplus Q'\to E\oplus Q'\to 0$ where $Q\oplus Q'$ and
$E\oplus Q'$ are permutation $G$-lattices, i.e.\ $[M]^{fl}$ is
permutation.

\end{proof}

\begin{theorem} \label{t4.10}
Let $G=\langle\sigma,\tau:\sigma^p=\tau^2=1,\tau\sigma\tau^{-1}=\sigma^{-1}\rangle\simeq D_p$ where $p$ is an odd prime number.
Let $M$ be a $G$-lattice such that $[M]^{fl}$ is a direct sum of $\bm{Z}$, $\bm{Z}[H]$,
 $V$, $Y_0$, $Y_1$, $Y_2$ (up to a direct summand of some permutation $G$-lattice).
If $K/k$ is a Galois field extension with $\fn{Gal}(K/k)\simeq G$,
then $K(M)^G$ is stably $k$-rational.
\end{theorem}

\begin{proof}
Follow the proof of (2) of Theorem \ref{t4.9}.
\end{proof}

\begin{theorem} \label{t4.11}
Let $G\simeq D_p$ where $p$ is a prime number. If $h_p^{+} =1$ and
$M$ is a $G$-lattice which is flabby and coflabby, then $M$ is
stably permutation.

In fact, when $h_p^{+} =1$ and $p$ is an odd prime number, then a
$G$-lattice is flabby (resp. flabby and coflabby) if and only if
it is stably permutation.
\end{theorem}

\begin{proof}
If $p$ is an odd prime number, then a $G$-lattice is flabby if and
only if it is both flabby and coflabby by Theorem \ref{t2.7}. Then
apply Theorem \ref{t4.9}.

If $p=2$, i.e. $G=C_2\times C_2$ is the Klein-four group, by
Colliot-Th\'el\`ene and Sansuc's result \cite[Proposition 4]{CTS},
a $G$-lattice which is flabby and coflabby is necessarily stably
permutation.
\end{proof}

\begin{remark}
A similar result also due to Colliot-Th\'el\`ene and Sansuc is
that, if $G$ is the quaternion group of order 8, an invertible
$G$-lattice is stably permutation \cite[R5, page 187]{CTS}. See
\cite{EM} for related results.
\end{remark}

\section{Steinitz classes}

Let $p$ be an odd prime number. Suppose that $h_p^+\ge 2$ and
$\c{A}$ is a non-principal ideal in
$R_0=\bm{Z}[\zeta_p+\zeta_p^{-1}]$. In this section we will show
that $V_{\c{A}}$, $(Y_0)_{\c{A}}$, $(Y_1)_{\c{A}}$,
$(Y_2)_{\c{A}}$ are not stably permutation $G$-lattices by
applying the Steinitz classes of these $G$-lattices.

Recall the integral representations of cyclic groups of prime order.

\begin{theorem}[{Diederichsen and Reiner \cite[page 729, Theorem 34.31; Sw1, page 74, Theorem 4.19]{Di,Re,CR}}] \label{t5.1}
Let $S=\langle\sigma\rangle\simeq C_p$ the cyclic group of prime
order $p$, $h_p$ be the class number of $\bm{Q}(\zeta_p)$. Let
$\c{B}$ range over a full set of representatives of the $h_p$
ideal classes of $\bm{Z}[\zeta_p]$. Then there are precisely
$2h_p+1$ isomorphism classes of indecomposable $S$-lattices, and
there are represented by
\[
\bm{Z},~ \c{B}
\]
and the non-split extensions
\[
0\to \c{B}\to W_{\c{B}} \to \bm{Z} \to 0.
\]
\end{theorem}

The $S$-lattices $W_{\c{B}}$ are rank-one projective modules over $\bm{Z}[S]$.

\begin{defn} \label{d5.2}
Let $R=\bm{Z}[\zeta_p]$ and $C(R)$ be the ideal class group of $R$
(written multiplicatively). For any $S$-lattice $N$, we define the
Steinitz class of $N$, denoted by $cl(N)$, by $cl(\bm{Z})=[R]$,
the equivalence class containing the principal ideal $R$,
$cl(\c{B})=[\c{B}]\in C(R)$, $cl(W_{\c{B}})=[\c{B}]\in C(R)$.
Furthermore, it satisfies the condition: If $0\to N'\to N\to N''
\to 0$ is an exact sequence of $S$-lattices, then
$cl(N)=cl(N')\cdot cl(N'')$ in $C(R)$ (see \cite[page 73; CR, page
729]{Sw1}). For any $S$-lattice $N$, its Steinitz class $cl(N)$ is
uniquely determined by $N$ \cite[page 75]{Sw1}.
\end{defn}

\begin{lemma}[{\cite[pages 78--80]{Sw1}}] \label{l5.3}
Let $S=\langle\sigma\rangle \simeq C_p$ where $p$ is a prime number,
$\Phi_p(X)=1+X+X^2+\cdots+X^{p-1}\in \bm{Z}[X]$ be the $p$-th cyclotomic polynomial.
For any $S$-lattice $N$, define $N_0=\{x\in N: \sigma(x)=x\}$, $N_1=\{x\in N:\Phi_p(\sigma)\cdot x=0\}$.
Then $N/N_0$ may be regarded as a module over $\bm{Z}[S]/\Phi_p(\sigma) \simeq \bm{Z}[\zeta_p]$.
Thus, as a module over $\bm{Z}[\zeta_p]$, $N/N_0\simeq \bigoplus_{1\le i\le t} I_i$ where each $I_i$ is an ideal in $\bm{Z}[\zeta_p]$.
The Steinitz class $cl(N)$ is equal to $[I_1\cdot I_2\cdot \cdots \cdot I_t] \in C(\bm{Z}[\zeta_p])$;
it is also equal to $cl(N_1)$.
\end{lemma}

\begin{proof}
Following the presentation of \cite[page 78]{Sw1}, we get the fibre product diagram
\[
\xymatrix{\bm{Z}[S] \ar[r] \ar[d] & \bm{Z}[S]/\langle\Phi_p(\sigma)\rangle\simeq \bm{Z}[\zeta_p] \ar[d] \\
\bm{Z}\simeq \bm{Z}[S]/\langle\sigma-1\rangle \ar[r] & \bm{Z}/p\bm{Z} }
\]

For any $S$-lattice $N$, $N/N_0$ is a lattice over
$\bm{Z}[S]/\Phi_p(\sigma)$, $N/N_1$ is a lattice over
$\bm{Z}[S]/\langle\sigma-1\rangle$. Moreover, we get the following
diagram
\[
\xymatrix{N \ar[r] \ar[d] & N/N_0 \ar[d] \\ N/N_1 \ar[r] & N/N_0+N_1}
\]
It follows that $N$ is isomorphic to the pull-back of $N/N_0$ and
$N/N_1$ along $N/\langle N_0+N_1 \rangle$. The Steinitz class
$cl(N)$ is uniquely determined by $N/N_0$ (see \cite[page
79]{Sw1}). In \cite[page 79]{Sw1}, $N/N_0$ is written as a
normalized form $\bigoplus_{1\le i\le t} I_i$ where $I_2\simeq
I_3\simeq \cdots \simeq I_t\simeq \bm{Z}[\zeta_p]$ and $I_1\simeq
\c{B}$ is a non-zero ideal of $\bm{Z}[\zeta_p]$.

The formula $cl(N)=cl(N_1)$ follows from the exact sequence $0\to N_1\to N\to N/N_1\to 0$ and $cl(N)=cl(N_1)\cdot cl(N/N_1)$ (see Definition \ref{d5.2}),
because $N/N_1$ is a lattice over $\bm{Z}[S]/\langle\sigma-1\rangle \simeq \bm{Z}$ and thus $N/N_1\simeq \bm{Z}^{(m)}$ for some integer $m$.
\end{proof}

\begin{theorem} \label{t5.4}
Let $G=\langle\sigma,\tau:\sigma^p=\tau^2=1,\tau\sigma\tau^{-1}=\sigma^{-1}\rangle\simeq D_p$ where $p$ is an odd prime number.
Assume that $h_p^+\ge 2$.
For any non-principal ideal $\c{A}$ of $\bm{Z}[\zeta_p+\zeta_p^{-1}]$,
let $M$ be one of the $G$-lattices $V_{\c{A}}$, $(Y_0)_{\c{A}}$, $(Y_1)_{\c{A}}$ or $(Y_2)_{\c{A}}$ in Theorem \ref{t4.1}.
Then $[M]^{fl}$ is not permutation.
\end{theorem}

\begin{proof}
Step 1. By Proposition \ref{p4.8} $M$ is flabby. Hence it is
invertible by Theorem \ref{t2.7}. Choose a $G$-lattice $M$ such
that $M\oplus N$ is a permutation $G$-lattice. Write $Q:=M\oplus
N$. It follows that $0\to M\to Q\to N\to 0$ is a flabby resolution
of $M$ and $[M]^{fl}=[N]$. We will show that $N$ is not stably
permutation.

Suppose not.
There is a permutation $G$-lattice $Q_1$ such that $N\oplus Q_1$ is a permutation $G$-lattice.
Write $Q_2:= N\oplus Q_1$.
Then $0\to N\to Q_2 \to Q_1\to 0$ is exact.

Let $S=\langle\sigma\rangle$ be the subgroup of $G$. By
restricting to the subgroup $S$, we may regard the exact sequences
of $G$-lattices $0\to M\to Q\to N\to 0$, $0\to N\to Q_2\to Q_1\to
0$ as exact sequences of $S$-lattices. We will find a
contradiction by evaluating the Steinitz classes of these
$S$-lattices.

\bigskip
Step 2.
Recall $R=\bm{Z}[\zeta_p]$.
If $Q_0$ is a permutation $S$-lattices, we will show that $cl(Q_0)=[R]$.

Since $|S|=p$ is a prime number, any permutation $S$-lattice is a direct sum of $\bm{Z}$ and $\bm{Z}[S]$.
In particular, $Q_0=\bm{Z}^{(s)} \oplus (\bm{Z}[S])^{(t)}$ for some non-negative integers $s$ and $t$.
Note that $cl(\bm{Z})=[R]$.
We will show that $cl(\bm{Z}[S])=[R]$ also.

Write $L:=\bm{Z}[S]$.
Define $L_0=\{x\in \bm{Z}[S]: (\sigma-1)\cdot x=0\}$.
It is easy to see that $L_0=\langle\Phi_p(\sigma)\rangle$ the ideal generated by $\Phi_p(\sigma)$.
Thus $L/L_0=\bm{Z}[S]/\Phi_p(\sigma)\simeq R$.
By Lemma \ref{l5.3} $cl(L)=[R]$. Done.

\bigskip
Step 3. From the exact sequence $0\to N\to Q_2\to Q_1\to 0$, we
find that $cl(N)=[R]$ because $[R]=cl(Q_2)=cl(N)\cdot
cl(Q_1)=cl(N)\cdot [R]$.

On the other hand, from the exact sequence $0\to M\to Q\to N\to
0$, we have $[R]=cl(Q)=cl(M)\cdot cl(N)=cl(M)\cdot [R]$. Thus
$cl(M)=[R]$. We will show that $cl(M)=[R]$ is impossible. Thus a
contradiction is obtained.

\medskip
Step 4.
Recall that $M=V_{\c{A}}$, $(Y_0)_{\c{A}}$, $(Y_1)_{\c{A}}$, $(Y_2)_{\c{A}}$.
We will consider the case $M=(Y_1)_{\c{A}}$;
the other cases  may be proved similarly.

By Theorem \ref{t4.1}, we have an exact sequence of $G$-lattices
$0\to P\c{A} \to (Y_1)_{\c{A}}\to \bm{Z}_-\to 0$. Regard it as an
exact sequence of $S$-lattices by restriction. When the action of
$\tau$ is forgotten, then $P\simeq R$, $\bm{Z}_-\simeq \bm{Z}$ as
$S$-lattices. Hence, as $S$-lattices, we have $0\to R\c{A}\to
(Y_1)_{\c{A}} \to \bm{Z}\to 0$ and
$cl((Y_1)_{\c{A}})=cl(R\c{A})\cdot cl(\bm{Z})=[R\c{A}]$.

By \cite[page 40, Theorem 4.14]{Wa}, the natural map of the class
group of $\bm{Z}[\zeta_p+\zeta_p^{-1}]$ to that of
$R=\bm{Z}[\zeta_p]$ is injective. Since we choose $\c{A}$ to be a
non-principal ideal of $\bm{Z}[\zeta_p+\zeta_p^{-1}]$, it follows
that $R\c{A}$ is a non-principal ideal of $R$. Thus
$cl(M)=cl((Y_1)_{\c{A}})=[R\c{A}]\ne [R]$ as we claimed before.
\end{proof}

\begin{proof}[Proof of Theorem \ref{t1.4}] ~

If $h_p^+=1$, then all the $D_p$-tori are stably rational by Theorem \ref{t4.9}.

If $h_p^+\ge 2$ let $G=D_p$, choose a $G$-lattice $M$ such that
$[M]^{fl}$ is not permutation by Theorem \ref{t5.4}. Let $T$ be
the $G$-torus defined over $k$ with character module $M$. By
Theorem \ref{t2.6} $T$ is not stably $k$-rational (but $T$ is
retract $k$-rational by Proposition \ref{p3.7}).
\end{proof}

\section{Some related rationality problems}

By Theorem \ref{t4.9}, if $h_p^+=1$, $M$ is any $D_p$-lattice and
$K/k$ is a Galois extension with $Gal(K/k) \simeq D_p$, then
$K(M)^{D_p}$ is stably $k$-rational. In this section we will
estimate the number of variables $m$ (which depends on $M$ and its
decomposition; see Lemma \ref{l6.4}) such that
$K(M)^{D_p}(x_1,\ldots,x_m)$ is $k$-rational. The key idea of our
method is the notion of anisotropic lattices exploited by
Voskresenskii and his school (see \cite{Ku1,Ku2}). Before the
proof, we recall two known rationality criteria.

\begin{theorem}[{\cite[Theorem 2.1]{Ka1}}] \label{t6.1}
Let $L$ be a field and $G$ be a finite group acting on $L(x_1,\ldots,x_m)$,
the rational function field of $m$ variables over $L$.
Suppose that

{\rm (i)}
for any $\sigma \in G$, $\sigma(L)\subset L$;

{\rm (ii)}
the restriction of the action of $G$ to $L$ is faithful;

{\rm (iii)}
for any $\sigma\in G$,
\[
\begin{pmatrix} \sigma(x_1) \\ \vdots \\ \sigma(x_m) \end{pmatrix}
=A(\sigma)\begin{pmatrix} x_1 \\ \vdots \\ x_m
\end{pmatrix}+B(\sigma)
\]
where $A(\sigma)\in GL_m(L)$ and $B(\sigma)$ is an $m\times 1$ matrix over $L$.
Then $L(x_1,\ldots,x_m)=L(z_1,\ldots,z_m)$ where $\sigma(z_i)=z_i$ for any $\sigma\in G$, any $1\le i\le m$.
In particular, $L(x_1,\ldots,x_m)^G=L^G(z_1,\ldots,z_m)$.
\end{theorem}

\begin{prop} \label{p6.2}
Let $G$ be a finite group, $M$ be a $G$-lattice. Let $k'/k$ be a
finite Galois extension such that there is a surjection $G\to
\fn{Gal}(k'/k)$. Suppose that there is an exact sequence of
$G$-lattices $0\to M_0\to M\to Q\to 0$ where $Q$ is a permutation
$G$-lattice. If $G$ is faithful on the field $k'(M_0)$, then
$k'(M)=k'(M_0)(x_1,\ldots,x_m)$ for some elements
$x_1,x_2,\ldots,x_m$ satisfying $m=\fn{rank}_{\bm{Z}} Q$,
$\sigma(x_j)=x_j$ for any $\sigma\in G$, any $1\le j\le m$.
\end{prop}

\begin{proof}
Note that the action of $G$ on $k'(M)$ is the purely
quasi-monomial action in Definition \ref{d2.3}.

Write $M_0=\bigoplus_{1\le i\le n} \bm{Z}\cdot u_i$,
$Q=\bigoplus_{1\le j\le m} \bm{Z}\cdot v_j$. Choose elements
$w_1,\ldots,w_m\in M$ such that $w_j$ is a preimage of $v_j$ for
$1\le j\le m$. It follows that $\{u_1,\ldots,
u_n,w_1,\ldots,w_m\}$ is a $\bm{Z}$-basis of $M$.

For each $\sigma\in G$, since $Q$ is permutation,
$\sigma(w_j)-w_l\in M_0$ for some $w_l$ (depending on $j$). In the
field $k'(M)$, if we write
$k'(M)=k(u_1,\ldots,u_n,w_1,\ldots,w_m)$ as the rational function
field in $m+n$ variables over $k'$, then
$\sigma(w_j)=\alpha_j(\sigma)w_l$ for some $\alpha_j(\sigma)\in
k'(M_0)$.

Since $G$ is faithful on $k'(M_0)$, apply Theorem \ref{t6.1}.
\end{proof}

\begin{defn}[Kunyavskii \cite{Ku1}] \label{d6.3}
Let $G$ be a finite group, $M$ be a $G$-lattice. $M$ is called an
anisotropic lattice if $M^G=0$ where $M^G:=\{x\in M:\sigma \cdot
x=x ~\forall \sigma\in G\}$. For a $G$-lattice $M$, define $M_0
:=\{x\in M:(\sum_{\sigma\in G}\sigma)\cdot x=0\}$. Then $M_0$ is
an anisotropic sublattice of $M$. Moreover, $(M/M_0)^G=M/M_0$
(for, if $\bar{x}\in M/M_0$ and $\sigma\in G$, then
$(\sigma-1)\cdot \bar{x}=0$, because $(\sigma-1)\cdot x\in M_0$).
\end{defn}

\begin{lemma} \label{l6.4}
Let
$G=\langle\sigma,\tau:\sigma^p=\tau^2=1,\tau\sigma\tau^{-1}=\sigma^{-1}\rangle
\simeq D_p$ where $p$ is an odd prime number. Let $M$ be any
$G$-lattice and $M_0=\{x\in M: (\sum_{\sigma\in G} g)\cdot x=0\}$.
If $h_p^+=1$, then $M_0\simeq X^{(s_0)} \oplus R^{(s_1)} \oplus
P^{(s_2)} \oplus \bm{Z}_-^{(t)}$ for some non-negative integers
$s_0$, $s_1$, $s_2$, $t$, which may not be uniquely determined by
the lattice $M_0$.
\end{lemma}

\begin{proof}
Note that $(\sum_{g\in G} g)\cdot M_0=0$, by definition. From
Theorem \ref{t4.9}, the only indecomposable $G$-lattices
annihilated by $\sum_{g\in G} g$ are $\bm{Z}_-$, $R$, $P$ and $X$.

Note that the Krull-Schmidt-Azumaya Theorem (see \cite[page
128]{CR}) is not valid in the category of $G$-lattices; for
examples, Theorem \ref{t3.4}, Theorem \ref{t3.5} and Theorem
\ref{t3.6} provide such counter-examples. Thus the integers $s_0$,
$s_1$, $s_2$, $t$ are not uniquely determined solely by the
lattice $M_0$.
\end{proof}

\begin{theorem} \label{t6.5}
Let the notations and assumptions be the same as in Lemma
\ref{l6.4}. Let $K/k$ be a Galois extension with $G\simeq
\fn{Gal}(K/k)$. Define $m=\fn{rank}_{\bm{Z}} M-\fn{rank}_{\bm{Z}}
M_0$, $n=s_0(p+1)+s_1(p+2)+s_2-t-m$, which may depends on the
decomposition of $M_0$ in Lemma \ref{l6.4}. If $n<0$, then
$K(M)^G$ is $k$-rational. If $n\ge 0$, then
$K(M)^G(z_1,z_2,\ldots,z_n)$ is $k$-rational where
$z_1,z_2,\ldots,z_n$ are elements algebraically independent over
$K(M)^G$.
\end{theorem}

\begin{proof}
Step 1. Note that $G$ acts trivially on $M/M_0$. By Proposition
\ref{p6.2}, $K(M)^G=K(M_0)^G (x_1,\ldots,x_m)$ where
$m=rank_{\bm{Z}}(M/M_0)$.

Note that $K(M_0)^G=K(X^{(s_0)}\oplus R^{(s_1)} \oplus P^{(s_2)})
(y_1,y_2,\ldots,y_t)$ where $\sigma\cdot y_i=y_i$, $\tau\cdot
y_i=1/y_i$ for $1\le i\le t$. Define $v_i=(1+y_i)/(1-y_i)$ if
$\fn{char}k\ne 2$, define $v_i=1/(1+y_i)$ if $\fn{char}k=2$. Then
$\sigma\cdot v_i=v_i$, $\tau\cdot v_i=-v_i$ or $v_i+1$ depending
on $\fn{char}k\ne 2$ or $\fn{char}k=2$. Apply Theorem \ref{t6.1},
we find that $K(M_0)^G=K(X^{(s_0)}\oplus R^{(s_1)} \oplus
P^{(s_2)})(w_1,\ldots,w_t)$ where $\sigma(w_i)=\tau(w_i)=w_i$ for
$1\le i\le t$.

It remains to add variables $z_1,\ldots,z_l$ such that
$K(X^{(s_0)}\oplus R^{(s_1)}\oplus P^{(s_2)})(z_1,\ldots,z_l)$ is
$G$-isomorphic to $K(N_1\oplus N_2\oplus\cdots\oplus N_r)$ where
each $N_i$ is a $G$-lattice satisfying the condition that
$K(N_1\oplus N_2\oplus\cdots\oplus N_d)^G$ is rational over
$K(N_1\oplus\cdots\oplus N_{d-1})^G$ for all $1\le d\le r$.

This condition may be fulfilled if (i) $N_d$ is a permutation
lattice by applying Theorem \ref{t6.1}, or (ii)
$\fn{rank}_{\bm{Z}} N_d=2$ by applying Voskresenskii's Theorem for
2-dimensional tori \cite[page 57]{Vo}, or (iii)
$\fn{rank}_{\bm{Z}} N_d=3$ and $N_d$ gives rise to a rational
torus in Kunyavskii list \cite[Theorem 1]{Ku3}. Once the lattices
$N_1,\ldots,N_r$ are found, we may show that
$K(N_1\oplus\cdots\oplus N_r)^G$ is $k$-rational inductively.
Hence $K(X^{(s_0)}\oplus R^{(s_1)} \oplus P^{(s_2)})^G
(z_1,\ldots, z_l)$ is $k$-rational. But we have $t+m$ variables
arising from $\bm{Z}_-^{(t)}$ and $M/M_0$. Thus $l-(t+m)$ extra
variables is required. This explains the definition of $n$ in the
statement of the theorem.

\bigskip
Step 2.
For simplicity we consider how many variables we should add to $K(X)$, $K(R)$, $K(P)$ to
 achieve the goal in Step 1.

Consider $K(P)$ first. By Theorem \ref{t4.1}, $0\to P\to V\to
\bm{Z}\to 0$. Thus $K(V)=K(P)(x)$ by Proposition \ref{p6.2}. By
Lemma \ref{l4.4}, $V\simeq M_+$ is a permutation $G$-lattice.
Hence one more variable is enough for $K(P)$.

Consider $K(X)$. By Lemma \ref{l4.4}, $X \simeq M_-$. By
Definition \ref{d3.3}, we have an exact sequence $0\to\ M_-\to
\widetilde{M}_-\to \bm{Z}_-\to 0$. Thus $K(\widetilde{M}_-)$ is
$G$-isomorphic to $K(X)(y)$ by Proposition \ref{p6.2}. By Theorem
\ref{t3.5}, $\widetilde{M}_-\oplus \bm{Z}[G/\langle \tau \rangle]$
is a permutation lattice. But $K(\widetilde{M}_- \oplus
\bm{Z}[G/\langle
\tau\rangle])=K(\widetilde{M}_-)(u_1,u_2,\ldots,u_p)$ by
Proposition \ref{p6.2}. Thus $p+1$ variables is required for
$K(X)$.

Consider $K(R)$. By Theorem \ref{t4.1}, we have $0\to R\to Y_0 \to
\bm{Z}[H]\to 0$. From Lemma \ref{l4.5}, $Y_0\simeq
\widetilde{M}_-$.

Use the fact $\widetilde{M}_-\oplus \bm{Z}[G/\langle\tau\rangle]$
is permutation again. Thus we need $p+2$ extra variable this time.

In summary, for $K(X^{(s_0)} \oplus R^{(s_1)}\oplus P^{(s_2)})$,
we need $s_0(p+1)+s_1(p+2)+s_2$ extra variables. Subtract the
$t+m$ variables which were obtained previously.
\end{proof}

The same method may be used to prove Theorem \ref{t1.1} and
Theorem \ref{t1.2}. For the convenience of the reader we indicate
some crucial steps of the proof because \cite{Ku2} has only the
Russian version. We emphasize that our proof is almost the same as
those given by Voskresenskii and Kunyavskii in \cite{Vo,Ku1,Ku2}.

\begin{proof}[Proof of Theorem \ref{t1.1} and Theorem \ref{t1.2}] ~

(A) Let $G=C_p$ (where $h_p=1$), $C_4$ or $S_3$ ($\simeq D_3$), $M$ be any $G$-lattice.
Suppose $K/k$ is a Galois extension with $G\simeq \fn{Gal}(K/k)$.
We will show that $K(M)^G$ is $k$-rational.

Define $M_0=\{x\in M: (\sum_{g\in G} g)\cdot x=0\}$. As in the
proof of Theorem \ref{t6.5}, it remains to show that $K(M_0)^G$ is
$k$-rational.

\begin{Case}{1} $G=\langle \sigma \rangle \simeq C_p$ with $h_p=1$. \end{Case}

By Theorem \ref{t5.1}, $M_0\simeq R^{(m)}$ where $R=\bm{Z}[\zeta_p]$.
Note that $K(R)=K(x_1,x_2,\ldots,x_{p-1})$ with $\sigma:x_1 \mapsto x_2\mapsto \cdots \mapsto x_{p-1} \mapsto 1/(x_1x_2\cdots x_{p-1})$.
Define
\begin{gather*}
y_0=1+x_1+x_1x_2+\cdots+x_1x_2\cdots x_{p-1}, \\
y_1=1/y_0,~ y_2=x_1/y_0,~ \ldots,~ y_{p-1}=x_1x_2\cdots x_{p-2}/y_0.
\end{gather*}

Then $K(x_1,x_2,\ldots,x_{p-1})=K(y_1,y_2,\ldots,y_{p-1})$ with
$\sigma: y_1\mapsto y_2\mapsto\cdots\mapsto y_{p-1} \mapsto
1-y_1-y_2-\cdots-y_{p-1}$. Hence
$K(y_1,\ldots,y_{p-1})^{\langle\sigma\rangle}=K(z_1,\ldots,z_{p-1})^{\langle\sigma\rangle}$
where $\sigma(z_i)=z_i$ for $1\le i\le p-1$ by Theorem \ref{t6.1}.
The case $m \ge 2$ can be proved similarly.

\begin{Case}{2} $G=\langle\sigma\rangle \simeq C_4$. \end{Case}

The indecomposable $G$-lattices are listed in \cite[page 64]{Vo}.
We choose only these lattices which are annihilated by
$1+\sigma+\sigma^2+\sigma^3$. They are the lattices listed below
\[
(-1), ~ \begin{pmatrix} 0 & -1 \\ 1 & 0 \end{pmatrix},~ \begin{pmatrix} 0 & -1 & -1 \\ 1 & 0 & 0 \\ 0 & 0 & -1 \end{pmatrix}.
\]

The first one gives rise to $K(x)$ with $\sigma(x)=1/x$. As
before, the action can be linearized by setting $y=(1+x)/(1-x)$ or
$1/(1+x)$ depending on $\fn{char}k\ne 2$ or $\fn{char}k=2$.

The second is a rank-two lattice.
Thus it is rational by Voskresenskii's Theorem of 2-dimensional tori \cite[page 57]{Vo}.

The third one is the kernel of the augmentation map $\bm{Z}[G]\to
\bm{Z}$ by \cite[page 65, line 7]{Vo}. It gives rise to
$K(x_1,x_2,x_3)$ with $\sigma: x_1\mapsto x_2\mapsto x_3\mapsto
1/(x_1x_2x_3)$. This action can be linearized by the same method
of Case 1.

\begin{Case}{3} $G=\langle\sigma,\tau:\sigma^3=\tau^2=1,\tau\sigma\tau^{-1}=\sigma^{-1}\rangle\simeq S_3$. \end{Case}

As in the proof of Theorem \ref{t6.5} with $p=3$, $M_0=X^{(s_0)}
\oplus R^{(s_1)} \oplus P^{(s_2)} \oplus\bm{Z}_-^{(t)}$. The
action for $\bm{Z}_-$ can be linearized as before. Since
$\fn{rank}_{\bm{Z}}R=\fn{rank}_{\bm{Z}} P=2$, Voskresenskii's
Theorem takes case of these situations; thus it is unnecessary to
add new variables to ensure $k$-rational. Since $X$ is a rank-3
$D_3$-lattice, we apply Kunyavskii's Theorem \cite{Ku3}. Thus no
new variables are needed and we find that $K(X)$ is $k$-rational.
Done.

\bigskip
(B) Let
$G=\langle\sigma,\tau:\sigma^2=\tau^2=1,\sigma\tau=\tau\sigma\rangle
\simeq C_2\times C_2$, $M$ be any $G$-lattice. Suppose that $K/k$
is a Galois extension with $G\simeq \fn{Gal}(K/k)$. If $K(M)^G$ is
stably $k$-rational (resp. retract $k$-rational), it is
$k$-rational. Consequently, if $[M]^{fl}$ is flabby and coflabby,
then $K(M)^G$ is $k$-rational.

Define $M_0=\{x\in M: (\sum_{g\in G} g)\cdot x=0\}$. It follows
that $K(M_0)^G$ is also stably $k$-rational (resp. retract
$k$-rational by \cite[Lemma 3.4]{Ka2}). Note that $M_0$ is a
direct sum of indecomposable $G$-lattices annihilated by
$\sum_{g\in G} g$. These ``special" indecomposable lattices were
enumerated by Kunyavskii in \cite[page 537--538]{Ku1}. Except for
two of them, the ranks of these lattices are $\le 2$. Hence we may
apply Voskresenskii's Theorem again. The remaining two lattices
are of rank 3: One is $N_1$ which is the kernel of the
augmentation map $\bm{Z}[G]\to \bm{Z}$, the other is
$N_2=\fn{Hom}(N_1,\bm{Z})$ (see \cite[page 540]{Ku1}). By
\cite{Ku3}, $K(N_1)^G$ is $k$-rational and $K(N_2)^G$ is not
retract $k$-rational.

Since $K(M_0)^G$ is retract $k$-rational, we find that $N_2$ will
not appear as a direct summand of $M_0$. Hence $M_0=M_1\oplus
M_2\oplus\cdots\oplus M_t$ where $M_i$ is either $N_1$ or is of
$\fn{rank}\le 2$. Thus $K(M_0)^G$ is $k$-rational.

Finally, when $[M]^{fl}$ is flabby and coflabby, apply Theorem
\ref{t4.11}.

\end{proof}

Note that Theorem \ref{t1.5} is a consequence of the following
theorem.

\begin{theorem} \label{t6.6}
Let $p$ is a prime number, $G=\langle\sigma\rangle\simeq C_p$, and
$k$ be a field admitting a $G$-extension.
\begin{enumerate}
\item[{\rm (1)}] If $T$ is a $G$-torus over $k$ which is stably
rational, then $T$ is rational.

 \item[{\rm (2)}]
$h_p=1$ if and only if all the $G$-tori over $k$ are stably
rational.

\end{enumerate}
\end{theorem}

\begin{proof}
(1) Working on the character module $M$ of $T$, it suffices to
show that, if $K/k$ is a Galois extension with
$\fn{Gal}(K/k)\simeq G$ and $K(M)^G$ is stably rational, then it
is $k$-rational.

Define $M_0=\{x\in M:(\sum_{g\in G} g)\cdot x=0\}$. By assumption,
$K(M_0)^G$ is stably $k$-rational. Since $M_0$ is annihilated by
$\sum_{g\in G}g$, $M_0$ is a direct sum of the ideals $\c{B}$'s by
Theorem \ref{t5.1} where $\c{B}$'s are ideals of
$\bm{Z}[\zeta_p]$. Without loss of generality, we may write
$M_0=\c{B}\oplus (\bm{Z}[\zeta_p])^{(m)}$ for some ideal $\c{B}$
and some non-negative integer $m$.

Since $K(M_0)^G$ is stably rational, the class $[M_0]^{fl}$ is
stably permutation. Because $[\bm{Z}[\zeta_p]]^{fl}$ is
permutation, we find that $[\c{B}]^{fl}$ is stably permutation.

As before, checking the Steinitz class, we find that $\c{B}$ is a
principal ideal. In other words, $M_0=\bm{Z}[\zeta_p]^{(m')}$ for
some integer $m'$. But then the action of $\sigma$ on $K(M_0)$ may
be linearized as in the proof of Case 1 of (A) for the proof of
Theorem \ref{t1.1} and Theorem \ref{t1.2}. Thus $K(M_0)^G$ is
$k$-rational.

(2) It remains to show that, if $h_p\ge 2$, then there is a
$G$-torus which is not stably rational.

Use the same method as in the proof of Theorem \ref{t5.4}. It
suffices to find a $G$-lattice $M$ such that $[M]^{fl}$ is not
stably permutation.

Since $h_p\ge 2$, there is a non-principal ideal $\c{B}$ in
$\bm{Z}[\zeta_p]$. Define $M=\c{B}$. If $[M]^{fl}$ is stably
permutation, then there exist permutation $G$-lattices $Q_1$ and
$Q_2$ such that $0\to M\to Q_1\to Q_2\to 0$ is exact. Hence the
Steinitz class $cl(M)=[\bm{Z}[\zeta_p]]$. However we know that
$cl(M)=[\c{B}]$ and $\c{B}$ is not a principal ideal. A
contradiction.

\end{proof}

\begin{remark}
Part (1) of the above theorem is just a special case of a more
general result. Let $k$ be a field admitting a $C_n$-extension and
$T$ be a $C_n$-torus over $k$. Thanks to the works of Endo and
Miyata, Voskresenskii, Chistov, Bashmakov and Klyachko (see
\cite[pages 62-63, 69-71]{Vo}), if $n=p^aq^b$ where $p,q$ are
prime numbers and $a,b$ are non-negative integers, then a
$C_n$-torus $T$ is stably $k$-rational if and if it is
$k$-rational.
\end{remark}

The proof of Theorem \ref{t6.6} may be adapted to solve another rationality problem.

\begin{defn} \label{d6.8}
Let $G$ be any finite group, $k$ be any field. Let $k(x_g:g\in G)$
be the rational function field in $|G|$ variables over $k$ with a
$G$-action via $k$-automorphism defined by $h\cdot x_g=x_{hg}$ for
any $h,g\in G$. Define $k(G):= k(x_g: g\in G)^G$ the fixed field.
Noether's problem asks whether $k(G)$ is $k$-rational.
\end{defn}

\begin{theorem} \label{t6.7}
Let $k$ be any field, $G=\langle \sigma,\tau:\sigma^n=\tau^2=1,\tau\sigma\tau^{-1}=\sigma^{-1}\rangle\simeq D_n$ where $n\ge 3$ is an odd integer.
Define an action of $G$ on the rational function field $k(x_1,x_2,\ldots,x_{n-1})$ through $k$-automorphisms defined by
\begin{align*}
\sigma &: x_1\mapsto x_2\mapsto \cdots \mapsto x_{n-1} \mapsto 1/(x_1x_2\cdots x_{n-1}), \\
\tau &: x_i \leftrightarrow x_{n-i}.
\end{align*}

Then $k(x_1,\ldots, x_{n-1})^G$ is stably $k$-rational if and only if $k(G)$ is stably $k$-rational.
\end{theorem}

\begin{proof}
We may write $k(x_1,\ldots,x_{n-1})^G=k(M)^G$ where $M$ is the
$G$-lattice $\bm{Z}[G/\langle\tau\rangle]$ (for the definition of
the field $k(M)$ with $G$ actions, see Definition \ref{d2.5}).
Note that $M$ is nothing but $N_+$ in Definition \ref{d3.2}.

By Lemma \ref{l4.5}, we find that $0\to N_+\to \widetilde{M}_-\to \bm{Z}[G/\langle\sigma\rangle]\to 0$ is an exact sequence of $G$-lattices.
By Proposition \ref{p6.2}, $k(\widetilde{M}_-)$ is $G$-isomorphic to $k(M)(y_1,y_2)$ where $\sigma(y_i)=\tau(y_i)=y_i$ for $1\le i\le 2$.

By Theorem \ref{t3.5},
$\widetilde{M}_-\oplus\bm{Z}[G/\langle\tau\rangle]\simeq
\bm{Z}[G]\oplus \bm{Z}$. Thus, by Proposition \ref{p6.2} again,
$k(\widetilde{M}_-)(z_1,\ldots,z_n)$ is $G$-isomorphic to
$k(\bm{Z}[G])(z_0)$ where $\sigma(z_i)=\tau(z_i)=z_i$ for $0\le
i\le n$.

Hence $k(M)^G$ is stably isomorphic to $k(M)^G$ with
$M=\bm{Z}[G]$, which is nothing but $k(G)$.
\end{proof}

\section*{Appendix}

The ideas in this appendix were communicated by Shizuo Endo to the
second-named author.

\medskip
First we recall some terminology in \cite{EM}.

Let $G$ be a finite group. Define an equivalence relation in the
set of $G$-lattices: Two $G$-lattices $M$ and $N$ are equivalent,
denoted by $M - N$, if, for any field $k$ admitting a
$G$-extension, for any Galois extension $K/k$ with $Gal (K/k)
\simeq G$, the fields $K(M)^G$ and $K(N)^G$ are stably isomorphic
over $k$, i.e. $K(M)^G (X_1, \ldots, X_m) \simeq K(N)^G (Y_1,
\ldots, Y_n)$ for some algebraically independent elements
$X_i,Y_j$.

\setcounter{theorem}{0}
\renewcommand{\thetheorem}{A\arabic{theorem}}
\begin{lemma} \label{lA1}
Let $G$ be a finite group, $M$ and $N$ be $G$-lattices. Then $M -
N$ if and only if $[M]^{fl} =[N]^{fl}$.
\end{lemma}

\begin{proof}
Using the same idea in the proof of \cite[Theorem 1.7]{Len}, it is
not difficult to show that $K(M)^G$ and $K(N)^G$ are stably
isomorphic over $k$ if and only if there exist exact sequences $0
\to M \to E \to P \to 0$ and $0 \to N \to E \to Q \to 0$ where $P$
and $Q$ are permutation $G$-lattices and $E$ is some $G$-lattice.
The latter condition is equivalent to $[M]^{fl} =[N]^{fl}$ by
\cite[Lemma 8.8]{Sw2}.
\end{proof}

\begin{defn} \label{dA2}
Let $G$ be a finite group. The commutative monoid $T(M)$ is
defined as follows. As a set, $T(M)$ is the set of all equivalence
classes $[M]$ under the equivalence relation ``$-$" defined above
where $M$ is any $G$-lattice and $[M]$ is the equivalence class
containing $M$. The addition in $T(M)$ is defined by $[M]+[N]=[M
\oplus N]$.

Recall the flabby class monoid $F_G$ defined in Section 2. By
Lemma \ref{lA1}, it is easy to see that $T(G) \to F_G$ is an
isomorphism by sending $[M]$ to $[M]^{fl}$.
\end{defn}

\begin{defn} \label{dA3}
 For a
finite group $G$, let $\Lambda$ be a $\bm{Z}$-order satisfying
$\bm{Z}[G] \subset \Lambda \subset \bm{Q}[G]$. We will define the
locally free class group of $\Lambda$ following \cite{EM1}. Let
$K_0(\Lambda)$ be the Grothendieck group of the category of
locally free $\Lambda$-modules of finite constant ranks. Define a
subgroup $C(\Lambda)$ of $K_0(\Lambda)$ by $C(\Lambda) =\{[M] -
n[\Lambda]: M \, is \,\,locally \,\,free \,\,of \,\,rank \,n,$
$where \, \,n \,\, runs \,\, over \,\, all \,\, positive \,\,
integers \}$. The group $C(\Lambda)$ is called the locally free
class group of $\Lambda$.

For an idele definition of $C(\Lambda)$, see \cite[page 219]{CR2}.
\end{defn}

\begin{theorem} {\rm (Endo and Miyata \cite{EM})} \label{tA4}
Let $G=C_n$ or $D_p$ where $n$ is any positive integer and $p$ is
an odd prime number. Then $T(G) \simeq C(\Omega_{\bm{Z}[G]})$
where $\Omega_{\bm{Z}[G]} $ is a maximal order in $\bm{Q}[G]$
containing $\bm{Z}[G]$,
\end{theorem}

\begin{remark}
The above theorem is just a special case of Theorem 3.3 in
\cite{EM} (see \cite{EM3} also). When $G$ is a cyclic group,
besides the proof given in \cite{EM}, Swan has another proof in
\cite{Sw3}. When $G$ is a non-cyclic group (see \cite[page 189,
line -15]{EM3}), similar arguments and similar exact sequences as
in \cite[page 96]{EM} may be applied to obtain a proof.
\end{remark}

The following theorem gives a generalization of Theorem
\ref{t1.5}.

\begin{theorem} \label{tA5}
Let $G=C_n$ where $n$ is a positive integer, and $k$ be a field
admitting a $G$-extension. Then $h_n=1$ if and only if all the
$G$-tori over $k$ are stably $k$-rational.
\end{theorem}

\begin{proof}
Apply Theorem \ref{tA4}. Since $T(G) \simeq F_G$, it remains to
show that $C(\Omega_{\bm{Z}[G]})=0$. When $G=C_n$, it is not
difficult to verify that the maximal order $\Omega_{\bm{Z}[G]}$ is
isomorphic to $\prod_{d | n} \bm{Z}[\zeta_d]$ \cite[page
243]{CR2}. Hence the result.
\end{proof}

\begin{theorem} \label{tA6}
Let $p$ be an odd prime number, $G=D_p$, and $k$ be a field
admitting a $G$-extension. Then $h_p^{+}=1$ if and only if all the
$G$-tori over $k$ are stably $k$-rational.
\end{theorem}

\begin{proof}
Apply Theorem \ref{tA4} again. We will show that
$C(\Omega_{\bm{Z}[G]})=0$.

Note that $C(\bm{Z}[G]) \to C(\Omega_{\bm{Z}[G]})$ is surjective
\cite[page 230, Theorem 49.25]{CR2}. Define $D(\bm{Z}[G])$ by the
exact sequence $0 \to D(\bm{Z}[G]) \to C(\bm{Z}[G]) \to
C(\Omega_{\bm{Z}[G]}) \to 0$ (see \cite[page 234]{CR2}).

By \cite[page 259, Theorem 50.25]{CR2}, we find that
$C(\Omega_{\bm{Z}[G]}) \simeq C(\bm{Z}[G]) \simeq C[\zeta_p +
\zeta_p^{-1}]$. Hence the result.
\end{proof}

\newpage
\renewcommand{\refname}{\centering{References}}

\end{document}